\documentclass[11pt, letterpaper]{amsart}

\usepackage{amsfonts}
\usepackage{amssymb}
\usepackage{graphicx}
\usepackage{amsmath}
\usepackage{latexsym}
\usepackage{amscd}
\usepackage{xypic}
\usepackage{mathrsfs}
\usepackage{enumitem}
\usepackage{braket}
\usepackage{color}
\usepackage{comment}
\usepackage{marginnote}
\usepackage[normalem]{ulem}

\usepackage{imakeidx}
\makeindex[title=Index of Notation]

\usepackage{hyperref}

\allowdisplaybreaks 

\numberwithin{equation}{section}
\newtheorem{theorem}{Theorem}[section]
\newtheorem{lemma}[theorem]{Lemma}
\newtheorem{corollary}[theorem]{Corollary}

\newtheorem{proposition}[theorem]{Proposition}

\newtheorem*{theorem*}{Theorem}

\theoremstyle{definition}
\newtheorem{definition}[theorem]{Definition}

\theoremstyle{remark}
\newtheorem{remark}[theorem]{Remark}
\newtheorem{claim}[theorem]{Claim}

\theoremstyle{definition}

\newcommand{\NN}{\mathbb{N}}

\newcommand{\QQ}{\mathbb{Q}}
\newcommand{\RR}{\mathbb{R}}
\renewcommand{\SS}{\mathbb{S}}

\newcommand{\cC}{\mathcal C}
\renewcommand{\cD}{\mathcal D}

\renewcommand{\cH}{\mathcal H}

\newcommand{\cM}{\mathcal M}
\newcommand{\cN}{\mathcal N}

\newcommand{\cQ}{\mathcal Q}
\renewcommand{\cR}{\mathcal R}
\newcommand{\cS}{\mathcal S}
\newcommand{\cT}{\mathcal T}

\newcommand{\cX}{\mathcal X}

\newcommand{\sF}{\mathscr{F}}

\newcommand{\sP}{\mathscr{P}}
\newcommand{\sS}{\mathscr{S}}

\newcommand{\bH}{\mathbf{H}}

\newcommand{\bx}{\mathbf{x}}
\newcommand{\by}{\mathbf{y}}
\newcommand{\bz}{\mathbf{z}}
\newcommand{\bOh}{\mathbf{0}}
\newcommand{\ba}{\mathbf{a}}
\newcommand{\bb}{\mathbf{b}}

\newcommand{\fG}{\mathfrak{G}}
\newcommand{\fF}{\mathfrak{F}}

\newcommand{\fN}{\mathfrak{N}}

\newcommand{\fS}{\mathfrak{S}}

\newcommand{\fp}{\mathfrak{p}}
\newcommand{\id}{\text{id}}

\DeclareMathOperator{\supp}{supp}

\newcommand{\eps}{\varepsilon}

\DeclareMathOperator{\sing}{sing}
\DeclareMathOperator{\reg}{reg}

\DeclareMathOperator{\spine}{spine}

\DeclareMathOperator{\singgen}{sing_\textnormal{gen}}
\DeclareMathOperator{\singng}{sing_\textnormal{non-gen}}

\title{Mean curvature flow with generic low-entropy initial data II}

\author{Otis Chodosh} 
\address{OC: Department of Mathematics, Bldg.\ 380, Stanford University, Stanford, CA 94305, USA}
\email{ochodosh@stanford.edu}

\author{Christos Mantoulidis} 
\address{CM: Department of Mathematics, Rice University, Houston, TX 77005, USA}
\email{christos.mantoulidis@rice.edu}
\author{Felix Schulze}
\address{FS: Department of Mathematics, Zeeman Building, University of Warwick, Gibbet Hill Road, Coventry CV4 7AL,
UK}
\email{felix.schulze@warwick.ac.uk} 

\begin{document}

\begin{abstract}
We prove that the mean curvature flow of a generic closed embedded hypersurface in $\mathbb{R}^4$ or $\mathbb{R}^5$ with entropy $\leq 2$, or with entropy $\leq \lambda(\SS^1)$ if in $\RR^6$, encounters only generic singularities. 
\end{abstract}

\maketitle

\section{Introduction}
Mean curvature flow is the gradient flow of area. A family of hypersurfaces $M(t) \subset \RR^{n+1}$ is flowing by mean curvature flow if the following equation is satisfied
\begin{equation} \label{eq:mcf}
\left(\tfrac{\partial}{\partial t} \bx \right)^{\perp} = \bH_{M(t)}(\bx).
\end{equation}
Here, $\bH_{M(t)}(\bx)$ denotes the mean curvature vector of $M(t)$ at $\bx$. When the initial data $M(0)$ is compact, mean curvature flow is guaranteed to become singular in finite time, so understanding the nature of such singularities is a fundamental problem. 

A well-known conjecture of Huisken suggests that the singularities of a generic mean curvature flow should be as simple as possible: spherical and cylindrical \cite[\#8]{Ilmanen:problems}. One approach to this issue is to study only the singularities that persist under a generic perturbation of the initial data $M(0)$. In this article we study this problem under a low-entropy condition (see Definition \ref{defi:entropy} for the definition of entropy). In particular, we obtain the following result (see Corollary \ref{coro:R4} for the precise statement)
\begin{theorem}[Low entropy generic flow in $\RR^4$, informal]\label{theo:infR4}
If $M^3 \subset \RR^4$ is a closed embedded hypersurface with entropy $\lambda(M) \leq 2$ then there exist arbitrarily small $C^\infty$ graphs $M'$ over $M$ so that the mean curvature flow starting from $M'$ has only multiplicity-one singularities of $\SS^3,\SS^2\times \RR$, and $\SS^1 \times \RR^2$-type. 
\end{theorem}

\begin{remark}
This implies that any mean curvature flow of $M'$ will be smooth for a.e.\ time $t$, see Corollary \ref{coro:reg} below.
\end{remark}

Together with K. Choi, we proved Theorem \ref{theo:infR4} for $M^2 \subset \RR^3$ with $\lambda(M) \leq 2$ in \cite[Theorem 1.1]{CCMS:low-ent-gen}. The generalization to $\RR^4$ presents significant new challenges. In particular, the flow of a closed embedded $M^3\subset \RR^4$ with $\lambda(M)\leq 2$ may have singularities modeled on \textit{singular} self-shrinkers, such as the cone over the Clifford torus.\footnote{Or more generally, this includes the cone over any of the infinitely many minimal submanifolds of $\SS^3$ with area less than $8\pi$.} Singular shrinkers were a major obstacle in many previous works (cf.\ $(\star_{n,\Lambda})$ in \cite{BernsteinWang:topology-small-ent,BernsteinWang:schoenflies} and $(\dagger_{n,\Lambda})$ in \cite{CCMS:low-ent-gen}), and precisely informed our previous strong entropy $\leq \lambda(\SS^1)$ assumption for $M^3\subset\RR^4$ in \cite[Theorem 1.2]{CCMS:low-ent-gen}).

We also have results analogous to Theorem \ref{theo:infR4} in higher dimensions (see Corollaries \ref{coro:R5} and \ref{coro:R6} for the precise statements):

\begin{theorem}[Low entropy generic flow in $\RR^5$, informal]\label{theo:infR5}
If $M^4 \subset \RR^5$ is a closed embedded hypersurface with entropy $\lambda(M) \leq 2$ then there exist arbitrarily small $C^\infty$ graphs $M'$ over $M$ so that the mean curvature flow starting from $M'$ has only multiplicity-one singularities of $\SS^4,\SS^3\times \RR$, $\SS^2 \times \RR^2$, and $\SS^1\times\RR^3$-type. 
\end{theorem}

\begin{theorem}[Low entropy generic flow in $\RR^6$, informal]
If $M^5 \subset \RR^6$ is a closed embedded hypersurface with entropy $\lambda(M) \leq \lambda(\SS^1)$ then there exist arbitrarily small $C^\infty$ graphs $M'$ over $M$ so that the mean curvature flow starting from $M'$ has only multiplicity-one singularities of $\SS^5,\SS^4\times \RR$, $\SS^3 \times \RR^2$, $\SS^2\times \RR^3$-type. 
\end{theorem}

Our results also imply a topological classification of low-entropy hypersurfaces (see Corollary \ref{cor:low-ent-top} for a more general statement and proof). 

\begin{corollary}[Mean curvature flow as a smooth isotopy in low entropy]\label{cor:low-ent-top-informal}
Let $n\in\{2, 3, 4, 5\}$ and $M^n\subset \RR^{n+1}$ be a closed connected embedded hypersurface.
\begin{enumerate}
\item[(a)] If $\lambda(M) \leq \lambda(\SS^{n-1})$ then perhaps after a small initial $C^\infty$ perturbation, the mean curvature flow provides a smooth isotopy from $M$ to a standard $\SS^n$ in $\RR^{n+1}$.
\item[(b)] If $\lambda(M) \leq \lambda(\SS^{n-2})$ then perhaps after a small initial $C^\infty$ perturbation, the mean curvature flow with surgery of \cite{Daniels-Holgate}  (cf.\ \cite{HuiskenSinestrari:surgery, HaslhoferKleiner:surgery}) provides a smooth isotopy from $M$ to the boundary of a standard handlebody that is either a standard ball $B^n$ or a boundary connect sum of finitely many $B^{n-1} \times \SS^1$'s.
\end{enumerate}
\end{corollary}

This topological classification is new when $n \in \{ 4, 5 \}$. When $n \in \{ 2, 3 \}$, a slightly different version of (a) was first proven by Bernstein--Wang \cite{BernsteinWang:schoenflies,BernsteinWang:TopologicalProperty}, and as presented (a) and (b) both follow from our work with K. Choi \cite{CCMS:generic1,CCMS:low-ent-gen}. Our work relied on various insights from the Bernstein--Wang program (see also \cite{ColdingMinicozziIlmanenWhite,BernsteinWang:1,Zhu:entropy,BernsteinWang:TopologicalProperty,BernsteinWang:relative-entropy,BernsteinWang:expander-compactness,BernsteinWang:degree-expander,bernsteinWang:top-uniqueness-expanders,HershkovitsWhite:sharp-entropy}). 

At the moment, serious obstacles remain to extend the isotopy construction (Corollary \ref{cor:low-ent-top-informal}) to the case $\lambda(M)\leq 2$ using Theorem \ref{theo:infR4}, even assuming a (hypothetical) construction of a flow with surgery (cf.\ \cite{haslhofer-du:hearing-shape,CHH:winglike,CHH:translator,DH:rotating,CDDHS:bubble-sheet-oval}).  

We refer the reader to \cite{HuiskenSinestrari:convexity,HuiskenSinnestrari:MCF-mean-convex,HuiskenSinestrari:surgery,Brendle:inscribed-sharp,BrendleHuisken:R3,HaslhoferKleiner:estimates,HaslhoferKleiner:surgery,BHH:2-convex-spheres, BuzanoHaslhoferHerskovits,ADS,ADS2,BrendleChoi:3d,BrendleChoi:nD,ChoiHaslhoferHershkovits,ChoiHaslhoferHershkovitsWhite,SunXue:AC,SunXue:closed,SunXue:cyl,ChuSun:genus1} for other results on mean curvature flow.

\subsection{The low entropy assumption and a strong multiplicity-one conjecture} As a consequence of our work with K.\ Choi \cite{CCMS:generic1,CCS:generic2} (building on \cite{Ilmanen:singularities,ColdingMinicozzi:generic,ColdingMinicozziIlmanenWhite,Brendle:genus0,Wang:ends-conical,BernsteinWang:TopologicalProperty}), the ``only'' remaining obstacle to a theory of generic mean curvature flows in $\RR^3$ without any entropy assumptions was the possible occurrence of smooth tangent flows with integer multiplicity---which the well-known ``multiplicity-one conjecture'' of Ilmanen \cite{Ilmanen:problems} posited should not occur. Some time after the first version of this paper appeared, Bamler--Kleiner confirmed the multiplicity-one conjecture (in $\RR^3$) \cite{BamlerKleiner}. When combined with our previous work, this verified Huisken's conjecture in $\RR^3$. More recently, we found an alternative ``entropy drop'' approach \cite{CCMS:genericR3} to replace the results in \cite{CCMS:generic1,CCS:generic2}.  

 In Appendix \ref{app:strong-mult-one}'s Theorem \ref{theo:generic-mult-one-assump}, we show how Theorems \ref{theo:infR4} and \ref{theo:infR5} can generalize to give generic mean curvature flows in $\RR^4$, $\RR^5$ under the assumption that a certain strong multiplicity-one property holds (without requiring $\lambda(M) \leq 2$).  See also Remark \ref{rema:strong.mult.one.r3}.

\subsection{Generic regularity of minimizing hypersurfaces} In \cite{CMS:hardt-simon910} we proved generic regularity for solutions to the Plateau problem (for hypersurfaces) in ambient $9$- and $10$-dimensions (see also \cite{CMS:improved-reg}). The present paper can be considered as a parabolic analogue of \cite{CMS:hardt-simon910}, but there are several new serious issues in this paper that were not present there.

In the elliptic problem, a crucial feature is that singular points cannot limit to regular points (by density considerations). In the parabolic problem, we do \textit{not} know in general whether non-generic singularities can limit to generic ones (but see \cite{ColdingIlmanenMinicozzi} for a related result); it is only known that this cannot happen if the limiting singularity is $\SS^n(\sqrt{2n})$ (then apply \cite{Huisken:convex}) or $\SS^{n-1}(\sqrt{2(n-1)}) \times \RR$ (then apply \cite{BernsteinWang:TopologicalProperty,ChoiHaslhoferHershkovits,ChoiHaslhoferHershkovitsWhite}). This creates serious issues when working with an entropy bound $\lambda(M) \leq 2$ (as we do in this paper) rather than $\lambda(M) \leq \lambda(\SS^{n-1})$ (as in \cite{CCMS:low-ent-gen}). We expect the techniques we have developed here to handle this (cf.\ Section \ref{sec:proof-main}) to be applicable in other situations. 

There are two additional serious issues in going from the elliptic to the parabolic problem:
\begin{enumerate}
\item The dimension of the spine of a singularity versus the eigenvalue estimate are \emph{coupled} (cf.\ Propositions \ref{prop:uniform-est-spine-eig} and \ref{prop:dens-fS-drop}), necessitating a more complicated covering argument to prove the final result. 
\item The parabolic Harnack inequality requires one to ``go back in time,'' significantly complicating the proof of the separation estimates (cf.\ Proposition \ref{prop:main-cover-sep}). 
\end{enumerate}

\subsection{Related work in other settings} 
We also mention here some other related work. Figalli--Ros-Oton--Serra \cite{FigalliRosOtonSerra:obstacle.generic.regularity} have proven generic regularity of solutions to the obstacle problem (see also \cite{FS:obst-fine}). They have subsequently applied these techniques to study the Stefan problem \cite{FROS:stefan}. These papers have some features in common with the present paper (and \cite{CMS:hardt-simon910,CMS:improved-reg}), particularly the general strategy of (super-linear) space-time H\"older estimates versus estimates for the space-time singular set, although the technical details are very different. In particular, the serious issue of ``generic'' versus ``non-generic'' singularities faced here does not  appear to have an analog in the previously cited works.

Finally, we also note the recent work of Fern\'andez-Real--Yu \cite{FRY:freebdry} establishing improved generic regularity for free-boundary problems by using a strategy that seems to have some relation with the technique used in \cite{CMS:hardt-simon910,CMS:improved-reg}.

\subsection{Definitions}
Before precisely stating our main results we first recall several definitions, some of which are now standard.

\begin{definition}[Colding--Minicozzi entropy] \label{defi:entropy}
	Following Colding--Minicozzi \cite{ColdingMinicozzi:generic}, we define the entropy of $M^n\subset \RR^{n+1}$ to be:
	\begin{equation*} %\label{eq:def-entropy}
	\lambda(M) := \sup_{\substack{\bx_0\in\RR^{n+1}\\t_0>0}}\int_{M} (4\pi t_0)^{-\frac n 2} e^{-\frac{1}{4t_0} |\bx - \bx_0|^{2}} \, d\cH^n(\bx)\, .
	\end{equation*}
	By Huisken's monotonicity of Gaussian area, $t\mapsto \lambda(M(t))$ is non-increasing when $M(t)$ is a family of closed hypersurfaces flowing smoothly by mean curvature flow, or more generally a Brakke flow with bounded area ratios (which holds, including for blowups, assuming that it holds for the initial data).
\end{definition}

\begin{remark} \label{rema:lambda.product}
	It is not hard to see that if $M^n \subset \RR^{n+1}$ is a closed hypersurface, then $\lambda(M) = \lambda(\bz_0 + O(M \times \RR^\ell))$ for all $\ell \geq 1$, $O \in O(n+\ell+1)$, $\bz_0\in\RR^{n+\ell+1}$. 
\end{remark}

\begin{definition}[Shrinkers] \label{defi:shrinkers}
	A (possibly incomplete) smooth hypersurface $\Sigma \subset \RR^{n+1}$ is said to be a shrinker if it satisfies $\bH + \tfrac12 \bx^\perp = \bOh$. By Huisken's monotonicity formula, these model the smooth parts of singularity models in mean curvature flow. We consider the following sets of shrinkers:
	\begin{enumerate}
		\item $\cS_n$, the set of \textit{complete} smooth self-shrinkers in $\RR^{n+1}$ with $\lambda(\Sigma) < \infty$.
		\item $\cS_n^*$, the set of non-flat elements of $\cS_n$.
		\item $\cS^\textrm{gen}_n$, the set of ``generic'' shrinkers, i.e., round self-shrinking spheres and cylinders over them in $\RR^{n+1}$:
			\[ \cS^\textrm{gen}_n : = \left\{O(\SS^{j}(\sqrt{2j}) \times \RR^{n-j}) \in \cS_n : j = 1,\dots,n, \, O \in O(n+1) \right\}. \]
	\end{enumerate}
	When working with entropy upper bounds, we need only consider low-entropy shrinkers. For $\Lambda > 0$, we define:
	\begin{align*}
		\cS_n(\Lambda) & := \{\Sigma \in \cS_n : \lambda(\Sigma) < \Lambda\}, \\
		\cS_n^*(\Lambda) & : = \cS_n(\Lambda) \cap \cS_n^*. 
	\end{align*}
\end{definition}

\begin{remark}[Generic shrinker entropies] \label{rema:stone}
	Note that Stone \cite{Stone} computed the entropies of spheres---and thus all generic shrinkers by Remark \ref{rema:lambda.product}---to be
	\[
		2 > \lambda(\SS^1)=\sqrt{\frac{2\pi}{e}} \approx 1.52  > \frac 32 > \lambda(\SS^2) = \frac{4}{e} \approx 1.47 > \dots > \lambda(\SS^n). 
	\]
\end{remark}

\begin{definition}[Forwards-backwards parabolic balls]
	Throughout the paper, for spacetime points $X=(\bx,t) \in \RR^{n+1}\times \RR$ and radii $r > 0$, 
	\[ P_r(X) = B_r(\bx) \times (t-r^2,t+r^2) \]
	denotes the forwards-backwards parabolic ball of radius $r$ about $X$.
\end{definition}

\begin{definition}[Regular vs singular and generic vs non-generic points]
Let $\cM$ be an integral $n$-Brakke flow in $\RR^{n+1}$ with support $\supp \cM$. 
\begin{enumerate}
	\item Let $\reg\cM$ be the set of all \emph{regular points} $X \in \supp \cM$, i.e., those for which there exists $r>0$ so that $\cM \lfloor P_r(X)$ is the multiplicity one Brakke flow associated to a smooth mean curvature flow. 
	\item Let $\sing\cM$ be the set of all \emph{singular points} $X \in \supp \cM\setminus \reg\cM$ so that $\cM$ is defined for times slightly before $t$.
	\item Decompose $\sing \cM$ into ``generic'' and ``non-generic'' singular points:
	\[
		\sing \cM = \sing_\textrm{gen}\cM \cup  \sing_\textrm{non-gen}\cM,
	\]
	where a point $X \in \sing \cM$ is declared to be in $\sing_\textrm{gen}\cM$ if some (and thus any by \cite{ColdingIlmanenMinicozzi,ColdingMinicozzi:uniqueness-tangent-flow}) tangent flow is the self-similar flow associated to $\Sigma \in \cS_n^\textrm{gen}$ with multiplicity one; otherwise $X \in \sing_\textrm{non-gen}\cM$.
\end{enumerate}  
Note that, by this definition, if $\cM$ is a Brakke flow starting at $t=0$, then we have the disjoint union decomposition
\[ \supp\cM = \supp\cM(0) \cup \reg\cM \cup \sing\cM. \]
\end{definition}

We now refer the reader to \cite[Section 2]{CCMS:generic1} for the complete definitions of the terms used below.  
\begin{definition}\label{defi:fF}
Let $M^n\subset \RR^{n+1}$ be a closed embedded hypersurface.
\begin{enumerate}
	\item[(a)] We denote by $\hat \fF(M)$ the set of unit-regular cyclic integral Brakke flows $\cM$ with $\cM(0) = \cH^n\lfloor M$. 
	\item[(b)] We denote by $\fF(M)\subset \hat\fF(M)$ the set of $\cM \in \hat \fF(M)$ with the following property: there exist closed embedded hypersurfaces $M_i\subset \RR^{n+1}$, each disjoint from $M$, but converging smoothly to $M$, and $\cM_i\rightharpoonup \cM$ for some  $\cM_i \in \hat \fF(M_i)$. 
\end{enumerate}
Note that $ \fF(M)\neq \emptyset$ by \cite{Ilmanen:levelset,White:cyclic} (see also \cite[Appendix B]{HershkovtisWhite}).
\end{definition}
 
\subsection{Main theorem and discussion of main hypotheses}
 
\begin{theorem}\label{thm:main}
Let $n \in \{ 2, 3, \ldots \}$, $\Lambda \in (0, 2]$, and assume that our two main hypotheses, \eqref{diamond_cond} and \eqref{heart_cond} (discussed below), both hold. 

Then, for every closed embedded hypersurface $M^n\subset \RR^{n+1}$ with $\lambda(M) \leq \Lambda$ there exist arbitrarily small $C^\infty$ graphs $M'$ over $M$ so that $\sing_\textnormal{non-gen} \cM' = \emptyset$ for all $\cM'\in\fF(M')$.
\end{theorem}

\begin{remark}
	The case $n=2$ already follows from \cite[Theorem 1.1]{CCMS:low-ent-gen}, but it is recovered by our more general techniques, so we include it for generality.
\end{remark}

\textit{First main hypothesis}, \eqref{diamond_cond}. Not every shrinker one encounters along a weak mean curvature flow is a complete smooth hypersurface, as considered in $\cS_n$. Instead, it may itself have singular points, which (after a blow-up argument) are modeled by minimal cones. In \textit{certain} occasions (see Proposition \ref{prop:finite-mu-smooth}) these cones are, in fact, stable. Together with a dimension reduction argument in geometric measure theory due to Federer, this naturally leads us to consider the following class of objects:

\begin{definition}[Regular stable cones]
	We define $\cR\cS\cC_k^*$ to be the set of regular stable stationary cones $\cC^k\subset \RR^{k+1}$, i.e., those cones that are stationary with $\sing \cC = \{\bOh\}$ (in particular, $\cC$ is non-flat), and so that $\reg \cC$ is stable as a minimal surface. 
\end{definition}
By \cite{Simons:minvar}, $\cR\cS\cC_k^* = \emptyset$ for $2 \leq k \leq 6$. On the other hand, $\cR\cS\cC_k^* \neq \emptyset$ for $k \geq 7$. As with shrinkers, we will only be interested in cones with entropy below our upper bound $\Lambda > 0$:
\[
	\cR\cS\cC^*_k(\Lambda) : = \{\cC \in \cR\cS\cC^*_k : \lambda(\cC) < \Lambda\}. 
\]
Our first main hypothesis is that there are \textit{no} such cones below our entropy threshold:
\begin{equation}\label{diamond_cond}
\tag{$\diamondsuit_{n,\Lambda}$} 
\cR\cS\cC_k^*(\Lambda) = \emptyset \qquad \textrm{ for } k=7,\ldots,n. 
\end{equation}

\begin{remark} \label{rema:diamond}
	Note that \eqref{diamond_cond} is vacuously true for $n \leq 6$ and $\Lambda > 0$. It is a well-known open problem to estimate $\Lambda > 0$ so that \eqref{diamond_cond} holds if $n \geq 7$ (cf.\ \cite{IlmanenWhite:sharp.entropy}).
\end{remark}

\textit{Second main hypothesis}, \eqref{heart_cond}. For a smooth but possibly incomplete shrinker $\Sigma$, we define $\mu(\Sigma)$ to be the first eigenvalue of the $L$-operator, namely 
\begin{equation}\label{eq:first-eig-L}
\mu(\Sigma) : = \inf\left\{   \int_\Sigma (|\nabla f|^2 - |A|^2 f^2 - \tfrac 12 f^2) e^{-\frac 14 |\cdot|^2} : f \in C^\infty_c(\Sigma), \int_\Sigma f^2 e^{-\frac 1 4 |\cdot|^2}  = 1  \right\},
\end{equation}
We then define 
\[
\cS_n^*(\Lambda,\mu) : = \{ \Sigma \in \cS_n^*(\Lambda) : \mu(\Sigma) \geq \mu\}. 
\]
The second main hypothesis is that low-entropy non-generic shrinkers are ``sufficiently'' unstable:
\begin{equation}\label{heart_cond}
\tag{$\heartsuit_{n,\Lambda}$} 
\cS_{k}^*(\Lambda,-\tfrac{n-k}{2})\subset \cS_k^\textrm{gen} \qquad \textrm{ for } k=2,\ldots,n-3. 
\end{equation}
\begin{remark} \label{rema:heart}
Note that \eqref{heart_cond} does hold for:
\begin{enumerate}
	\item[(a)] $n \leq 4$ and all $\Lambda$, vacuously;
	\item[(b)] $n = 5$ and $\Lambda \leq \lambda(\SS^1)$, since $\cS_2^*(\lambda(\SS^1),-\frac 32) \subset  \cS_2^*(\lambda(\SS^1))\subset \cS_2^\textrm{gen}$ by \cite[Corollary 1.2]{BernsteinWang:TopologicalProperty}.
\end{enumerate}
In our proof, $(\heartsuit_{n,\Lambda})$ is used in Proposition \ref{prop:uniform-est-spine-eig} to ensure that $\mu$ (our primary tool for instability) is sufficiently negative to overcome the possible presence of a spine along which non-generic singularities may accumulate.
\end{remark}

\subsection{Applications of main theorem}

The following are all immediate corollaries of our main theorem and Remarks \ref{rema:diamond}, \ref{rema:heart}. We skip the low dimensional case $n=2$ since it was already stated in \cite{CCMS:low-ent-gen}. Everywhere below, $M$ denotes a closed embedded hypersurface. 

\begin{corollary}\label{coro:R4}
For $M^3\subset \RR^4$ with $\lambda(M) \leq 2$ there exist arbitrarily small $C^\infty$ graphs $M'$ over $M$ so that $\sing_\textnormal{non-gen} \cM' = \emptyset$ for all $\cM'\in\fF(M')$.
\end{corollary}

This improves the previous bound $\lambda(M) \leq \lambda(\SS^1)$ from \cite{CCMS:low-ent-gen}.

\begin{corollary}\label{coro:R5}
For $M^4\subset \RR^5$ with $\lambda(M) \leq 2$ there exist arbitrarily small $C^\infty$ graphs $M'$ over $M$ so that $\sing_\textnormal{non-gen} \cM' = \emptyset$ for all $\cM'\in\fF(M')$.
\end{corollary}

\begin{corollary}\label{coro:R6}
For $M^5\subset \RR^6$ with $\lambda(M) \leq \lambda(\SS^1)$ there exist arbitrarily small $C^\infty$ graphs $M'$ over $M$ so that $\sing_\textnormal{non-gen} \cM' = \emptyset$ for all $\cM'\in\fF(M')$. 
\end{corollary}

These two corollaries had no analog in \cite{CCMS:low-ent-gen}.

\begin{remark} \label{rema:heart.52}
	To relax the entropy bound $\lambda(\SS^1)$ to $2$ in Corollary \ref{coro:R6}, one could try to see whether or not $(\heartsuit_{5,2})$ holds. This is an independently interesting question: we are asking if all $\Sigma \in \cS_2^*\setminus \cS_2^\textrm{gen}$ with entropy $\lambda(\Sigma)< 2$ have $\mu(\Sigma) < -\tfrac 32$. Note that numerics \cite{BK:bounds,BK:entropy,BK:index} suggest that Angenent's torus has $\lambda \sim 1.85$, $\mu \sim -3.74$.
\end{remark}

Here is a topological application of our theorem: 

\begin{corollary}\label{cor:low-ent-top}
Take $n \in \{ 2, 3, \ldots \}$, $\Lambda \in (0, 2]$ such that \eqref{diamond_cond} and \eqref{heart_cond} hold. Let  $M^n\subset \RR^{n+1}$ be a closed connected embedded hypersurface with $\lambda(M)\leq \Lambda$. 
\begin{enumerate}
\item[(a)] If $\Lambda \leq \lambda(\SS^{n-1})$ then perhaps after a small initial $C^\infty$ perturbation, the mean curvature flow provides a smooth isotopy from $M$ to a standard $\SS^n$ in $\RR^{n+1}$.
\item[(b)] If $\Lambda \leq \lambda(\SS^{n-2})$ then perhaps after a small initial $C^\infty$ perturbation, the mean curvature flow with surgery of \cite{Daniels-Holgate} provides a smooth isotopy from $M$ to the boundary of a standard handlebody that is either a standard ball $B^n$ or a boundary connect sum of finitely many $B^{n-1} \times \SS^1$'s.
\end{enumerate}
\end{corollary}
By Remarks \ref{rema:diamond} and \ref{rema:heart}, Corollary \ref{cor:low-ent-top} holds unconditionally (cf.\ Corollary \ref{cor:low-ent-top-informal}) when $n \leq 5$. See also the remarks after Corollary \ref{cor:low-ent-top-informal}.

\begin{proof}
Combining the argument used in the beginning of Section \ref{sec:proof-main} with Theorem \ref{thm:main}, either $M$ is a round sphere (in which case we are done), or we can find a small $C^\infty$ graph $M'$ over $M$ so that $\lambda(M) < \Lambda$ and so that there is $\cM'\in\fF(M')$ with $\singng\cM'=\emptyset$. 

Then, in case (a), we have $\lambda(\cM') < \lambda(\SS^{n-1})$ so $\cM'$ can only have multiplicity-one $\SS^n$-type singularities. Thus, the mean curvature flow is completely smooth until it becomes a round sphere. In case (b), we additionally use the surgery of \cite{Daniels-Holgate}. 
\end{proof}

Here is an application to regularity of the flow (by convention, $\lambda(\SS^{-1}) = \lambda(\SS^0) = 2$):
\begin{corollary}\label{coro:reg}
Take $n\in\{2,\dots\}$ and $\Lambda \in (0, \lambda(\SS^{n-3})]$ such that \eqref{diamond_cond} and \eqref{heart_cond} hold. Let  $M^n\subset \RR^{n+1}$ be a closed connected embedded hypersurface with $\lambda(M)\leq \Lambda$. There is an arbitrarily small graph $M'$ over $M$ so that any mean curvature flow $\cM'\in\fF(M')$ is completely smooth for a.e.\ time $t$ and each connected component of $\sing\cM'$ is contained in some time-slice. 
\end{corollary}
\begin{proof}
Combine Theorem \ref{thm:main} with \cite{ColdingMinicozzi:sing-generic}. 
\end{proof}

By Remarks \ref{rema:diamond} and \ref{rema:heart}, Corollary \ref{coro:reg} holds unconditionally when $n \leq 5$.

\subsection{Organization}
Section \ref{sec:prelim} contains some preliminary discussions about rescaled flows, regularity, smoothly crossing flows, and parabolic covering. Section \ref{sec:low-ent-F-var} contains regularity and eigenvalue estimates for low-entropy $F$-stationary varifolds. In Section \ref{sec:one-sided-flows} we establish the main local estimates for nearly ancient one-sided flows near nearly-self-similar flows. We then use this in Section \ref{sec:density-drop} to prove the central density drop result. We then show how to handle the issue of non-generic singularities limiting to bubble-sheet generic singularities in Section \ref{sec:proof-main}, completing the proof of the main result. Appendix \ref{app:strong-mult-one} discusses a form of the multiplicity-one conjecture as it relates to our results here. In Appendix \ref{app:stable-mu} we estimate the first eigenvalue of the $L$-operator on non-flat stable cones (this is not used elsewhere in the paper). 

\subsection{Acknowledgements} 
We are grateful to Kyeongsu Choi and Brian White for helpful discussions.  We are grateful to the referees for their careful reading and many helpful suggestions. O.C. was supported by a Terman Fellowship and an NSF grant (DMS-2304432). C.M. was supported by an NSF grant (DMS-2147521). 

\section{Preliminaries and notation} \label{sec:prelim}

\subsection{$F$-functional and density}

Let us begin by recalling a few basic notions in the smooth setting. First, for smooth hypersurfaces $M^n \subset \RR^{n+1}$ one has the $F$-functional,
\begin{equation} \label{eq:f-functional}
	F(M) = (4\pi)^{-\frac{n}{2}} \int_M e^{-\tfrac14 |\cdot|^2} \, d\cH^n,
\end{equation}
which in turn gives rise to the Colding--Minicozzi entropy (Definition \ref{defi:entropy}) via
\begin{equation} \label{eq:f-functional-entropy}
	\lambda(M) = \sup_{\substack{\bx_0\in\RR^{n+1}\\t_0>0}} F \left( \frac{1}{\sqrt{t_0}} (M - \bx_0) \right).
\end{equation} 
Moreover, for a smooth mean curvature flow with bounded area ratios $\cM : t \mapsto M(t)$, the $F$-functional also gives rise to the density function, defined as
\begin{equation} \label{eq:density-r}
	\Theta_\cM(X, r) = F\left( \frac{1}{r} (M(t - r^2) - \bx) \right), \; X \in \RR^{n+1} \times \RR, \; r > 0,
\end{equation}
which is nondecreasing in $r$ by Huisken's well-known monotonicity formula. In particular,
\begin{equation} \label{eq:density}
	\Theta_\cM(X) = \lim_{r \to 0} \Theta_\cM(X, r), \; X \in \RR^{n+1} \times \RR,
\end{equation}
is well-defined.

For our proofs, we will need to work with weaker objects than smooth hypersurfaces, namely $n$-varifolds in $\RR^{n+1}$ (see \cite{Simon:GMT}) and Brakke flows (see \cite{Ilmanen:elliptic} and \cite[Section 2]{CCMS:generic1}). In this setting, \eqref{eq:f-functional}, \eqref{eq:f-functional-entropy} extend to $n$-varifolds in $\RR^{n+1}$ (by looking at their induced Radon measures), and the monotonicity underlying \eqref{eq:density-r}, \eqref{eq:density} extends to $n$-Brakke flows in $\RR^{n+1}$ with bounded area ratios (cf.\ \cite[Lemma 7]{Ilmanen:singularities}). 

\subsection{$F$-stationary varifolds}

It is a simple computation that all elements of $\cS_n$ are stationary for the $F$-functional (cf.\ \cite[\S 3]{ColdingMinicozzi:generic}) or equivalently for the area functional with respect to the conformal metric $e^{-|\cdot|^2/(2n)} g_{\RR^{n+1}}$ on $\RR^{n+1}$. Thus, Definition \ref{defi:f-stationary-varifold} below precisely generalizes the notion of shrinkers to the varifold setting.

\begin{definition} \label{defi:f-stationary-varifold}
	An $n$-varifold $V$ in $\RR^{n+1}$ is called $F$-stationary if is stationary with respect to the conformally metric $e^{-|\cdot|^2/(2n)} g_{\RR^{n+1}}$ on $\RR^{n+1}$.
\end{definition}

\begin{definition}\label{defi:shrinking-flow}
Let $V$ be an $F$-stationary $n$-varifold in $\RR^{n+1}$. We define an associated $n$-Brakke flow $\cM_V$ as follows: if $V$ is a cone we define $\cM_V$ as a static flow
\[ \cM_V(t) = V \text{ for all } t \in \RR, \]
while if $V$ is not a cone, we define $\cM_V$ as a shrinking flow that disappears at $t=0$:
\[ \cM_V(t) = \begin{cases} \sqrt{-t}V & \text{ for } t<0, \\ 0 & \text{ for } t\geq0. \end{cases} \]
In the latter case $\cM_V$ may only be unit-regular for times $t < 0$; e.g., if $V$ is a non-conical self-shrinker with asymptotically conical ends, then $\cM_V$ vanishes ``unexpectedly'' at $t=0$. This will not pose any issue in the sequel; cf. the $T$ parameter in Lemma \ref{lemm:smooth-cross}.
\end{definition}

It is not hard to see, using the monotonicity formula on $\cM_V$, that
\[ \lambda(V) = F(V) \]
for all $F$-stationary $n$-varifolds $V$ in $\RR^{n+1}$.

\subsection{Regularity scale of Brakke flows} 
We use here the notation
\[ P_r(\bx,t) = B_r(\bx)\times (t-r^2,t+r^2) \]
for the forwards-backwards parabolic ball of radius $r$ at $(\bx,t)$.

\begin{definition}[Regularity scale] \label{defi:reg-scale}
	Let $\cM$ be an integral $n$-Brakke flow in $\RR^{n+1}$ and $X = (\bx,t) \in \reg\cM$. We define the \emph{regularity scale} of $\cM$ at $X$ by
	\[
		r_\cM(X) : = \sup\{r >0: P_r(X) \cap \supp \cM \subset \reg\cM, \; |A| \leq r^{-1} \textrm{ on }P_r(X) \}.
	\]
	For $\rho,r>0$ and $(\bx,t) \in \RR^{n+1}\times \RR$ we set 
	\[
	\cR_{\rho,r}\cM(\bx,t) : = \left\{\by  \in \supp \cM(t) \cap B_r(\bx) : r_\cM(\by,t) > \rho\right\} \subset \reg\cM(t). 
	\]
	For an integral $F$-stationary varifold we abuse notation and write $r_V(\bx) := r_{\cM_V}(\bx,-1)$ and $\cR_{\rho,r} V(\bx) : = \cR_{\rho,r} \cM_V(\bx,-1)$. 
\end{definition}

\begin{remark} \label{rema:reg-scale}
	For $\lambda>0$, the rescaled flows $\cM_\lambda(t) = \lambda\cM(\lambda^{-2}t)$ satisfy
	\[
	 \cR_{\lambda \rho,\lambda r} \cM_\lambda(\lambda\bx,\lambda^2 t) = \lambda \cR_{ \rho, r}\cM(\bx, t).
	\]
	Furthermore, for $\rho'\leq \rho,r\leq r'$ we have
	\[
	\cR_{\rho,r}\cM(\bx,t) \subset \cR_{\rho',r' + |\bx-\bx'|} \cM(\bx',t).
	\]
\end{remark}

\begin{remark}[Continuity of regularity scale] \label{rema:reg-scale-continuity}
	It is standard to show, under the assumption that the entropy is uniformly bounded above by $2-\varepsilon$ with $\varepsilon > 0$, that $r_\cM(X)$ is continuous with respect to convergence of the point $X$ and the $n$-Brakke flow $\cM$ if, e.g., the initial  conditions converge smoothly. See \cite[Lemma 2.4]{CMS:hardt-simon910} for the analogous result for minimizing hypersurfaces (note that the upper bound on the entropy prevents issues with multiplicity, which is not a concern for minimizing hypersurfaces). 
\end{remark}

\subsection{Rescaled Brakke flows} 
\begin{definition} \label{defi:rescaled-flow}
Given a Brakke flow $\cN$ and a spacetime point $X_0=(\bx_0,t_0)$ we can define the \emph{rescaled Brakke flow around $X_0$ at scale $\beta$} by the usual recipe
\[
\tilde \cN(\tau) : = \beta^{-1} e^{\tau/2} (\cN(t_0-\beta^2 e^{-\tau}) - \bx_0).
\]
Note that $\tau = 0$ corresponds to $t=t_0-\beta^2$ (if $\cN$ is defined then).
\end{definition}

\begin{definition}\label{defi:rescaled-brakke-regularity-scale}
Consider a rescaled Brakke flow $\tilde\cN$. We define the regularity scale $r_{\tilde \cN}(\bx,\tau)$ as follows. Define a non-rescaled flow
\[
\cN(t) = \sqrt{-t} \tilde\cN(\tau - \log(-t)) 
\]
and then set $r_{\tilde\cN}(\bx,\tau) := r_{\cN}(\bx,-1)$. Similarly, we define 
\[
\cR_{\rho,r}\tilde \cN(\bx,\tau) : =  \{ \by \in \supp \tilde\cN(\tau) \cap B_r(\bx): r_{\tilde\cN}(\by,\tau) > \rho\} .
\] 
Note that with this definition, $\cR_{\rho,r}\tilde \cN(\bx,\tau) = \cR_{\rho,r} \cN(\bx,-1)$.
\end{definition}

\subsection{Regularity of low-entropy $F$-stationary varifolds} 

\begin{lemma}\label{lemm:sing-F-stat}
Let $V$ be a cyclic integral $F$-stationary $n$-varifold in $\RR^{n+1}$ with $\lambda(V) < 2$. Then:
\begin{enumerate}
	\item[(a)] We have $\dim_H \sing V \leq n-3$. 
	\item[(b)] If $\sing V \neq \emptyset$, then $V$ has an iterated tangent cone of the form (up to rotation)
	\[ \RR^{n-k} \times \cC^k, \; k \geq 3, \]
	with $\cC^k\subset \RR^{k+1}$ a smooth multiplicity-one non-flat minimal cone.
\end{enumerate}
\end{lemma}
\begin{proof}
The assumption $\lambda(V) < 2$ rules out tangent cones given by higher multiplicity planes or unions four or more half-planes. The union of three half-planes is ruled out by the cyclic assumption thanks to \cite{White:cyclic}. The assertion then follows from standard dimension reduction arguments (cf.\ \cite{White:stratification}). \end{proof}

\subsection{Smoothly crossing Brakke flows} 

We recall here 
\begin{definition}[{\cite[Definition B.1]{CCMS:low-ent-gen}}] \label{defi:cross-smoothly}
Consider integral $n$-Brakke flows $\cM$, $\cM'$ in $\RR^{n+1}$. We say that $\cM$, $\cM'$ \emph{cross smoothly} at $X=(\bx,t) \in \reg \cM \cap \reg \cM'$ if there is $r>0$ so that $\cM\lfloor P(X,r),\cM'\lfloor P(X,r)$ are equal to smooth connected multiplicity-one flows $\Gamma(s),\Gamma'(s)$ so that $\Gamma(t)$ has points on both sides of $\Gamma'(t)$ in any small neighborhood of $\bx$. We also say that $\cM $ and $\cM'$ cross smoothly (with no reference to a point) if they cross smoothly at some $X \in \reg \cM \cap \reg \cM'$.
\end{definition}

\begin{lemma}[{\cite[Lemma B.2]{CCMS:low-ent-gen}}]\label{lemm:smooth-cross}
Consider integral $n$-Brakke flows $\cM_i$, $\cM_i'$ in $\RR^{n+1}$ that are unit-regular for times $t < T$ (possibly $T = \infty$). Assume that $\cM_i \rightharpoonup \cM$, $\cM'_i \rightharpoonup \cM'$, and that $\cM$, $\cM'$ cross smoothly at $X = (\bx, t)$ with $t < T$. Then, there are $X_j\to X$ so that $\cM_j$, $\cM_j'$ cross smoothly at $X_j$ for all sufficiently large $j$.
\end{lemma}
\begin{proof}
	The proof in the reference (where no $T$ parameter was present) still applies.
\end{proof}

\begin{lemma}\label{lemm:fF-dontsmoothlycross}
Let $M, M' \subset \RR^{n+1}$ be closed embedded hypersurfaces which are disjoint or coincide. Then, no pair $\cM \in \fF(M)$, $\cM' \in \fF(M')$ can cross smoothly.
\end{lemma}
\begin{proof}
If $M \neq M'$, then $M \cap M' = \emptyset$ implies $\supp\cM \cap \supp\cM'= \emptyset$ by the avoidance principle (cf.\ \cite[Lemma 10.6]{Ilmanen:elliptic}).

It remains to consider $M = M'$. Using the definition of $\fF(M)$, pick any $M_i$ converging smoothly to $M$ with $M_i\cap M = \emptyset$, and its corresponding $\cM_i \in \hat\fF(M_i)$ converging weakly to $\cM$. Do the same with $M'$, $M'_i$, $\cM'_i$. Passing to a subsequence of $\{ M'_i \}$, we may assume that $M_i \cap M'_i = \emptyset$, and thus $\supp \cM_i \cap \supp \cM'_i = \emptyset$ by avoidance. The result follows from Lemma \ref{lemm:smooth-cross}. 
\end{proof}

Abusing notation, we'll say that two $F$-stationary $n$-varifolds $V$, $V'$ in $\RR^{n+1}$ cross smoothly if their associated shrinking Brakke flows $\cM_V$, $\cM_{V'}$ cross smoothly in the sense of Definition \ref{defi:cross-smoothly}. Equivalently, $V$, $V'$ cross smoothly at $\bx$ if there exists $r>0$ so that $V\lfloor B_r(\bx)$, $V'\lfloor B_r(\bx)$ are  smooth connected multiplicity-one self-shrinkers $\Sigma$, $\Sigma'$ so that $\Sigma$ has points on both sides of $\Sigma'$ in any small neighborhood of $\bx$. 

\begin{lemma}\label{lemm:smooth-cross-stat-var}
If $V$, $V'$ are cyclic integral $F$-stationary $n$-varifolds in $\RR^{n+1}$ with $\lambda(V)$, $\lambda(V') < 2$, then either $V=V'$ or $V$, $V'$ cross smoothly.
\end{lemma}
\begin{proof}
By the Frankel property for self-shrinkers (cf.\ \cite[Corollary C.4]{CCMS:generic1}) we have $\supp V\cap\supp V' \not = \emptyset$. Suppose that $\reg V \cap \reg V' = \emptyset$. Since
\begin{equation}\label{eq:sing-set-small-smooth-cross}
\dim_H\sing V, \; \dim_H\sing V'\leq n-3
\end{equation} 
by Lemma \ref{lemm:sing-F-stat}, this would imply that $\dim_H(\supp V\cap \supp V') \leq n-3$. This contradicts the strong maximum principle for varifolds \cite[Theorem A(i)]{Ilmanen:strong-max} (cf.\ \cite{Wickramasekera:smp}). 

We may thus consider $\bx \in \reg V\cap\reg V'$. If $V$, $V'$ cross smoothly at $\bx$, we are done. If they do not, then $V$ lies (weakly) to one-side of $V'$ near $\bx$. The strong maximum principle then implies that $V=V'$ near $\bx$. Since $\reg V,\reg V'$ are connected (this holds by combining the Frankel property for self-shrinkers with \cite[Theorem A(ii)]{Ilmanen:strong-max}) 
we thus see that $\reg V =\reg V'$. Combined with \eqref{eq:sing-set-small-smooth-cross} and $\lambda(V)$, $\lambda(V') < 2$, we thus see that  $V=V'$. 
\end{proof}

\begin{corollary}\label{coro:dont-cross-dens-well-defined}
Let $M\subset \RR^{n+1}$ be a closed embedded hypersurface with $\lambda(M) < 2$. Suppose that $\cM, \cM' \in \fF(M)$ and $X \in \supp \cM \cap \supp\cM' \cap\{t>0\}$. Then:
\begin{enumerate}
	\item[(a)] $\Theta_\cM(X) = \Theta_{\cM'}(X)$, and 
	\item[(b)] all tangent flows at $X$ for $\cM$, $\cM'$ coincide for $t < 0$.
\end{enumerate}
\end{corollary}
\begin{proof}
Note that (b) implies (a), so we can just prove (b). To that end, it suffices to show that any tangent flow $\tilde \cM$ to $\cM$ at $X$ is also a tangent flow to $\cM'$ at $X$. The result then follows by swapping $\cM$ and $\cM'$.

Let $\tilde \cM$ be a tangent flow to $\cM$ at $X$ obtained by parabolically dilating by $\lambda_j \to \infty$. Let the corresponding dilated flows of $\cM$ be $\tilde \cM_j$, and those of $\cM'$ be $\tilde \cM'_j$, so that
\[ \tilde \cM_j \rightharpoonup \tilde \cM \]
and, after passing to a subsequence,
\[ \tilde \cM'_j \rightharpoonup \tilde \cM', \]
for a tangent flow $\tilde \cM'$ of $\cM'$ at $X$. By Lemma \ref{lemm:fF-dontsmoothlycross}, $\cM$ and $\cM'$ cannot cross smoothly, and thus neither can their dilations $\tilde \cM_j$, $\tilde \cM'_j$. Thus, by Lemma \ref{lemm:smooth-cross}, $\tilde \cM$, $\tilde \cM'$ cannot cross smoothly either. 

If $V$ and $V'$ are the $F$-stationary $n$-varifolds associated with $\tilde \cM$ and $\tilde \cM'$ then $\lambda(M) < 2$ implies that $\lambda(V)$ and $\lambda(V') < 2$, so in particular $V = V'$ by Lemma \ref{lemm:smooth-cross-stat-var}. Thus, the set of tangent flows to $\cM$ at $X$ is the same as the set of tangent flows to $\cM'$ at $X$, at least in their $t < 0$ portions that are uniquely determined by $V$, $V'$.
\end{proof}

\subsection{Parabolic planes and a covering lemma} 

We will be relying on estimates on the parabolic Hausdorff dimension of the singular set of a mean curvature flow obtained in \cite{White:stratification,CheegerHaslhoferNaber}. These estimates, as in the case of minimizing hypersurfaces, rely on the notion of the \textit{spine} of a tangent flow (see Definition \ref{defi:spine} below). To that end, we introduce some notation for the set of possible spines.

\begin{definition}\label{defi:spacetime-spine}
We denote $\mathscr{S}$ by the set of subspaces $\Pi$ of $\RR^{n+1}\times \RR$ that are either of the form $\Pi=\Pi' \times \{0\}$ or $\Pi'\times \RR$ for $\Pi'\subset \RR^{n+1}$ a subspace. Write $D(\Pi) = \dim \Pi'$ in the first case and $D(\Pi) = \dim \Pi' +2$ in the second.  This is just the \emph{parabolic} dimension of $\Pi\subset \RR^{n+1}\times\RR$  and satisfies $0 \leq D(\Pi) \leq n+2$.
\end{definition}

\begin{remark} \label{rema:no-quasistatic}
	In the general parabolic setting spines need not be elements of $\mathscr{S}$ due to the possibility of quasi-static cones arising as tangent flows. However, unit-regular flows with entropy $ < 2$ do not have quasi-static cones arising as tangent flows, so Definition \ref{defi:spacetime-spine} precisely suffices for our consideration of spines.
\end{remark}

For $\Pi \subset \RR^{n+1} \times \RR$, $\gamma > 0$, we will write
\[ U_\gamma(\Pi) : = \bigcup_{(\bx,t) \in \Pi} (B_\gamma(\bx) \times (t-\gamma^2,t+\gamma^2)) \]
for the parabolic (forward-backward) $\gamma$-tubular neighborhood of $\Pi$.

\begin{lemma}[{cf.\ \cite[(3.14)]{CheegerHaslhoferNaber}}] \label{lemm:covering}
There is $C=C(n)\in [1,\infty)$ with the following property. 

For any $\Pi\in \sS$ and $\gamma\in(0,1)$, if $\mathfrak{X} \subset U_\gamma(\Pi) \cap P((\bOh,0),1)$ is an arbitrary nonempty subset, then there are points $X_1,\dots,X_K \in \mathfrak{X}$ so that
\[ \mathfrak{X} \subset \cup_{k=1}^K P(X_k,\gamma) \text{ with } K \leq C\gamma^{-D(\Pi)}. \]
\end{lemma}
\begin{proof}
	This follows in a standard way from the Vitali covering lemma.
\end{proof}

\section{Low-entropy $F$-stationary varifolds} \label{sec:low-ent-F-var}

In this section we consider a cyclic integral $F$-stationary $n$-varifold $V$ in $\RR^{n+1}$ with $\lambda(V) < 2$. We will write $\mu(V) = \mu(\reg V)$ (see \eqref{eq:first-eig-L}).

\begin{proposition}\label{prop:finite-mu-smooth}
If $\mu(V) > -\infty$ then any blow-up limit of $V$ has stable regular part. 
\end{proposition}
\begin{proof}
Consider a blow-up limit $\tilde V_i = \lambda_i(V-\bx_i) \rightharpoonup \tilde V$ with $\lambda_i\to\infty$, $\bx_i\to \bx_\infty$. Assume, for contradiction, that $\tilde V$ has unstable regular part. Let $U$ be an open set such that
\[ \sing \tilde V \cap U = \emptyset \text{ and } \reg \tilde V \cap U \text{ is unstable}. \]
Fix some $\delta > 0$ and a nonzero $\tilde  f \in C^\infty_c(\reg \tilde V \cap U)$ with 
\[
\int_{\reg\tilde V} \big( |\nabla \tilde  f(\by)|^2 - |A(\by)|^2 \tilde f(\by)^2 \big) \, d\by\leq - 2\delta \int_{\reg\tilde V} \tilde  f(\by)^2 \, d\by.
\]
Since $\sing \tilde V \cap U = \emptyset $ we can find non-zero $\tilde  f_i \in C^\infty_c(\reg\tilde  V_i \cap U)$ converging to $f$ smoothly. Observing that as $i\to\infty$, 
\[
e^{\tfrac 1 4|\bx_i|^2}e^{-\tfrac 1 4|\bx_i+\lambda_i^{-1}\by|^2} = e^{-\tfrac 12 \lambda_i^{-1} \bx_i\cdot\by - \tfrac 1 4 \lambda_i^{-2} |\by|^2} \to 1 \text{ uniformly for $|\by|$ bounded},
\]
we thus conclude that, for $i$ sufficiently large,
\[
\int_{\reg \tilde  V_i} \big(|\nabla\tilde   f_i|(\by)^2 - |A(\by)|^2 \tilde  f_i(\by)^2 \big)e^{-\tfrac 1 4|\bx_i+\lambda_i^{-1}\by|^2} \, d\by \leq - \delta \int_{\reg \tilde  V_i} \tilde  f_i(\by)^2 e^{-\tfrac 1 4|\bx_i+\lambda_i^{-1}\by|^2} \, d\by.
\]
Define $f_i \in C^\infty_c(\reg V)$ by pulling $\tilde f_i$ back to the original scale, i.e.~$f_i(\bx) = \tilde f_i (\lambda_i(\bx-\bx_i))$. We have:
\begin{multline*}
\int_{\reg V} \big( |\nabla f_i|(\bx)^2 - |A(\bx)|^2 f_i(\bx)^2 - \tfrac 12 f_i(\bx)^2 \big) e^{-\tfrac 14 |\bx|^2} \, d\bx\\  \leq (- \delta\lambda_i^2 - \tfrac 12) \int_{\reg V} f_i(\bx)^2  e^{-\tfrac 14 |\bx|^2} \, d\bx.
\end{multline*}
Since $\lambda_i\to\infty$ we find that $\mu(V) = -\infty$. This completes the proof. 
\end{proof}

Arguing as in the proof of Lemma \ref{lemm:sing-F-stat} we have the following result. 

\begin{corollary}\label{coro:finite-mu-smooth}
Suppose that \eqref{diamond_cond} holds. If $\lambda(V) < \Lambda$ and $\mu(V) > -\infty$ then $\sing V = \emptyset$, i.e., $V$ is the varifold associated to some smooth self-shrinker $\Sigma \in \cS_n$.
\end{corollary}

We recall Definition \ref{defi:shrinking-flow} for the shrinking flow $\cM_V$ associated to $V$.  

\begin{definition}\label{defi:spine} For any $F$-stationary varifold $V$ we set
\begin{equation}\label{eq:spacetime-spine-uniform-est}
\spine V := \{X \in \RR^{n+1} \times \RR : \cM_V + X =\cM_V\}. 
\end{equation}
It follows from \cite[\S 8]{White:stratification} and Remark \ref{rema:no-quasistatic} that $\spine V \in \sS$ and  $\Theta_{\cM_V}(X) = F(V)$ if and only if $X \in \spine V$. By abuse of notation we will write
\[ D(V) = D(\spine V), \]
where $D(\spine V)$ is as in Definition \ref{defi:spacetime-spine}.
\end{definition} 

\begin{lemma}\label{lemm:spine-limiting}
Suppose that $V_j \rightharpoonup V$ is a convergent sequence of $F$-stationary integral varifolds. Then:
\begin{enumerate}
	\item[(a)] $\limsup_j D(V_j) \leq D(V)$, and 
	\item[(b)] $\limsup_j \mu(V_j) \leq \mu(V)$. 
\end{enumerate}
\end{lemma}
\begin{proof}
Suppose that $\cM_{V_j} + X_j = \cM_{V_j}$ for $X_j \in \spine V_j$. Without loss of generality, $|X_j| = 1$ for all $j$ and thus $X_j\to X$ after passing to a subsequence. Passing $\cM_{V_j} + X_j = \cM_{V_j}$ to the limit, too, we get $\cM_V + X = \cM_V$, so $X \in \spine V$.  This proves (a).

Assertion (b) follows from \eqref{eq:first-eig-L} since on every open set $U \Subset \RR^{n+1}\setminus \sing V$ we have smooth convergence of $V_j \to V$.
\end{proof}

The following estimate\footnote{With some extra effort, it should be possible to show that there is equality in Lemma \ref{lemm:prod-mu}, but we will not need this fact here. 
} is straightforward to prove by cutting off in the $\RR^\ell$-factors and using the finiteness of Gaussian area. 
\begin{lemma}\label{lemm:prod-mu}
For $\Sigma \in \cS_k$ we have $\mu(\RR^\ell \times \Sigma) \leq\mu(\Sigma)$. 
\end{lemma}

The following result is a crucial ingredient in this paper. We estimate the size of the spine versus the first eigenvalue of the $L$ operator. It is crucial that we prove a uniform bound, and this necessarily complicates the statement and proof. 

\begin{proposition}\label{prop:uniform-est-spine-eig}
Fix $n \in \{ 2,3,\ldots \}$, $\Lambda \in (0, 2]$ so that \eqref{heart_cond} and \eqref{diamond_cond} hold and let $\eps>0$. There are constants $\kappa=\kappa(n,\Lambda,\eps)>0,\rho_0=\rho_0(n,\Lambda,\eps)>0,r_0=r_0(n,\Lambda,\eps)> 2 $ so that the following holds. 

Let $V$ be a non-flat and non-generic $F$-stationary cyclic integral $n$-varifold in $\RR^{n+1}$ with
$\lambda(V) \leq \Lambda-\eps$. Then:
\begin{enumerate}
	\item[(a)] $\cR_{3\rho_0,r_0/3} V(\bOh)\neq \emptyset$. 
	\item[(b)] For any $\cR_{3\rho_0,r_0/3}V(\bOh) \subset \Omega\subset \reg V$, we have
	\[
	2\mu(\Omega) + D(V) < - 2\kappa.
	\]
\end{enumerate}
\end{proposition}
\begin{proof}
Suppose for contradiction that (a) fails, i.e., there is a sequence of  $F$-stationary cyclic integral $n$-varifolds in $\RR^{n+1}$, $V_j$, with
\[
 \cR_{3j^{-1},j/3}V_j(\bOh) = \emptyset .
\]
Pass to a subsequential limit $V_j\rightharpoonup V$. Note that $V$ is a non-flat $F$-stationary cyclic integral $n$-varifold with $\lambda(V) \leq \Lambda -\eps$ (by a cutoff argument) and $V \not \in \cS_n^\textrm{gen}$ (by \cite{ColdingIlmanenMinicozzi}). In particular, Lemma \ref{lemm:sing-F-stat} gives that $\cR_{\rho,r}V(\bOh) \neq \emptyset$ for $\rho$ sufficiently small and $r$ sufficiently large. This is a contradiction, proving (a).

Note that (a) remains valid if we decrease $\rho_0$ and increase $r_0$. As such, to prove (b) it suffices to consider $V_j$ with $\cR_{3j^{-1},j/3}V_j (\bOh) \subset \Omega_j \subset \reg V_j$ so that
\begin{equation}\label{eq:limiting-spine-est-contr}
0 \leq \limsup_j(2\mu(\Omega_j) + D(V_j)) 
\end{equation}
Pass to a subsequential limit $V_j\rightharpoonup V$ as above.  

\begin{claim} \label{claim:eig-dim-ineq-fixed-V} 
We have $2\mu(V) + D(V) < 0$. 
\end{claim}

Assume this claim for now. Choose $\rho>0$, $r>2$, $\Omega' \subset \cR_{4\rho,r/4}V (\bOh)\subset \reg V$ so that
\[
2\mu(\Omega') + D(V) < 0.
\]
Allard's theorem yields for sufficiently large $j$ regions $\Omega_j' \subset \cR_{3\rho,r/3}V_j(\bOh)\subset \Omega_j\subset  \reg V$ converging as smooth graphs over $\Omega'$. Transplanting test functions on $\Omega'$ to $\Omega_j'$ we find that $\limsup_j \mu(\Omega_j') \leq \mu(\Omega')$. Combined with $\mu(\Omega_j) \leq \mu(\Omega_j')$ and Lemma \ref{lemm:spine-limiting} we find
\[
\limsup_j(2 \mu(\Omega_j) + D(V_j))  \leq \limsup_j(2 \mu(\Omega_j') + D(V_j)) \leq 2\mu(\Omega') + D(V) < 0.
\]
This contradicts \eqref{eq:limiting-spine-est-contr}, and proves (b).

It remains to prove Claim \ref{claim:eig-dim-ineq-fixed-V}. Note it trivially holds if $\mu(V) = -\infty$, so by Corollary \ref{coro:finite-mu-smooth} we can take $V$ to be the varifold associated to a smooth self-shrinker $\Sigma \in \cS_n^* \setminus \cS_n^\textrm{gen}$. Then, $\mu(\Sigma) < -1$ by \cite[Theorems 9.36 and 0.17]{ColdingMinicozzi:generic}, so Claim \ref{claim:eig-dim-ineq-fixed-V} automatically holds if $D(V)\leq 2$. We can thus consider $D(V) \geq 3$. We know that 
\[ \Sigma = \Pi'\times \Sigma' \text{ for } \Sigma' \in \cS_{n-D(V)}^*\setminus\cS_{n-D(V)}^\textrm{gen}. \]
Set $k:=n-D(V)$ and note that $2\leq k \leq n-3$ (the first inequality follows from  $\cS_1^* = \cS_1^\textrm{gen}$ and the second from $D(V) \geq 3$). Since $\Sigma' \in \cS_k^*(\Lambda)\setminus \cS_k^\textrm{gen}$, \eqref{heart_cond} implies 
\[
2 \mu(\Sigma') < -n+k = - D(V). 
\]
Lemma \ref{lemm:prod-mu} implies that $\mu(\Sigma) \leq \mu(\Sigma')$, yielding Claim \ref{claim:eig-dim-ineq-fixed-V}. 
\end{proof}

\section{One-sided flows}\label{sec:one-sided-flows}

\subsection{Crossing between non-generic shrinkers and their translates}

The following result is \cite[Proposition 2.2]{CCMS:low-ent-gen} generalized to the current setting of singular tangent flows. Recall Definition \ref{defi:shrinking-flow} for the notion of associated shrinking flows.

\begin{proposition}\label{prop:geo-input}
Take $n \in \{ 2, 3, \ldots \}$, $\Lambda \in (0, 2]$ such that \eqref{diamond_cond} holds.

Let $V$, $V'$ be cyclic integral $F$-stationary $n$-varifolds in $\RR^{n+1}$ with $\lambda(V)$, $\lambda(V') < \Lambda$ and associated shrinking flows $\cM_V,\cM_{V'}$. If:
\begin{itemize}
	\item $V$ is non-flat and non-generic, and
	\item there is $X' \in \RR^{n+1}\times \RR$ so that $\cM_V$, $\cM_{V'} + X'$ do not cross smoothly,
\end{itemize}
then:
\begin{enumerate}
	\item[(a)] $V'=V$, and 
	\item[(b)] $X' \in \spine V$.
\end{enumerate}
\end{proposition}
\begin{proof}
Note that the tangent flow at $-\infty$ to $\cM_V$ is $\cM_V$ and the tangent flow at $-\infty$ to $\cM_{V'}+X'$ is $\cM_{V'}$. By Lemma \ref{lemm:smooth-cross} and definition of smoothly crossing $F$-stationary varifolds, $V$ and $V'$ do not cross smoothly. Lemma \ref{lemm:smooth-cross-stat-var} thus implies that $V=V'$. 

It remains to prove that $X'\in \spine V$. We may assume that $X' \neq (\bOh, 0)$ without loss of generality, since otherwise (b) is trivial. Write $X'=(\bx',t')$. Since $\cM_V$ is invariant under parabolic dilation, we note that $\cM_V$ and $\cM_V + (\lambda \bx',\lambda^2 t')$ do not cross smoothly for any $\lambda>0$. By considering $t=-1$, we see that the speed of the family
\[
\lambda\mapsto \sqrt{1+\lambda^2 t'} \reg V + \lambda \bx'
\]
at $\lambda=0$ must not change sign. Up to changing the sign of the unit normal $\nu$ on $\reg V$ (recalling $\reg V$ is connected) we thus find that $\bx' \cdot \nu\geq 0$ on  $\reg V$. Since $\bx'\cdot\nu$ satisfies the elliptic PDE $L (\bx'\cdot \nu) = \frac 12 \bx'\cdot \nu$ along $\reg V$ (cf.\ \cite[Theorem 5.2]{ColdingMinicozzi:generic}), the strong maximum principle implies that either $\bx' \cdot \nu > 0$ on $\reg V$ or $\bx' \cdot \nu \equiv 0$ on $\reg V$. 

We begin by considering the first case $\bx' \cdot \nu > 0$ on $\reg V$. By the argument in \cite[Lemma 2.1]{Wang:graph} we find that $\mu(V) \geq -\tfrac 12$ (recalling that $\mu(V)$ is the first eigenvalue of the $L$ operator on $\reg V$). Thus Corollary \ref{coro:finite-mu-smooth} implies that $\sing V= \emptyset$. Hence, by \cite[Theorem 1.1]{Wang:graph}, $V$ is a hyperplane. This contradicts that $V$ is non-flat.

Thus, we find that $\bx'\cdot \nu \equiv 0$. We first assume that $\bx'\neq 0$. This implies that $\reg V$ splits a line in the $\bx'$ direction. By Lemma \ref{lemm:sing-F-stat}, we see that $V$ splits a line in the $\bx'$ direction. If $t'=0$, we see that $X'\in\spine V$. If $t'\neq 0$ then since $V$ splits a line in the $\bx'$ direction, we see that $\cM_V + (\bOh,t') = \cM_V + (\bx',t')$ does not cross $\cM_V$ smoothly.

As such, it remains to consider the case that $\bx'=0$ and $t'\neq 0$. In this case, we can argue as above (by considering $t=-1$) see that the speed of
\[
\lambda \mapsto (1+\lambda t') \reg V 
\]
does not change sign on $\reg V$. This implies that (up to changing the sign of the unit normal) that the mean curvature $H\geq 0$ on $\reg V$. As above, we recall that
\[
L H = H
\] 
along $\reg V$, so the strong maximum principle (and connectedness of $\reg V$) yields either $H>0$ or $H\equiv 0$. As above, if $H>0$ then $\sing V= \emptyset$ (by arguing as in \cite[Lemma 2.1]{Wang:graph} and using Lemma \ref{lemm:sing-F-stat}). In this case, by \cite[Theorem 0.17]{ColdingMinicozzi:generic} (cf.\ \cite{Huisken:sing}) $V\in\cS_n^\textrm{gen}$, a contradiction. Thus, $H\equiv 0$ along $\reg V$. By Lemma \ref{lemm:sing-F-stat} we thus see that $V$ is a stationary cone. Thus, $(0,t')\in\spine V$. This completes the proof.  
\end{proof}

\subsection{Nearly self-similar non-generic flows}

Fix $n \in \{2, 3, \ldots\}$, $\Lambda \in (0, 2]$, $\eps > 0$, and consider a class of nearly self-similar flows about a non-generic singular point:

\begin{definition}\label{defi:Peta0} For $\eta \in (0, 1]$ define $\sP(\eta)$ to be the set of pairs $(X,\cM)$ so that:
\begin{enumerate}
\item $X \in P((\bOh,0),1)$,
\item $\cM$ is a unit-regular cyclic integral $n$-Brakke flow in $\RR^{n+1}$ for $t \in [-\eta^{-2},\infty)$, 
\item $\lambda(\cM) \leq \Lambda -\eps$, 
\item $X \in \sing_\textnormal{non-gen}\cM$, 
\item $\Theta_\cM(X) \geq \Theta_\cM(X,\eta^{-1})-4\eta$. 
\end{enumerate}
\end{definition}

\begin{remark} \label{rema:monotonicity-nearly-constant-contr-arg}
It is useful to observe that for all $(X,\cM) \in \sP(\eta)$, $r \in (0, \eta^{-1}]$,
\begin{equation*} 
\Theta_{\cM}(X) \geq \Theta_{\cM}(X,r) - 4\eta.
\end{equation*}
This follows from monotonicity and (5).
\end{remark}

\begin{lemma}\label{lemm:limit-cF-flows}
Fix $\Lambda \in (0, 2]$, $\eps > 0$. Suppose that $((\bOh,0),\cM_j) \in \sP(j^{-1})$ and $\cM_j\rightharpoonup \cM$. Then, there exists a non-flat and non-generic $F$-stationary cyclic integral $n$-varifold $V$ with $\lambda(V) \leq\Lambda-\eps$, and $\cM = \cM_V$ whenever the latter is nonvanishing (see Definition \ref{defi:shrinking-flow}).
\end{lemma}
\begin{proof}
It is clear that $\cM$ is a unit-regular cyclic integral $n$-Brakke flow in $\RR^{n+1}$ defined for $(-\infty,\infty)$. Due to Remark \ref{rema:monotonicity-nearly-constant-contr-arg} and upper-semicontinuity of density
\[
\Theta_\cM((\bOh,0)) \geq \Theta_\cM((\bOh,0),r)
\]
for all $r>0$. Thus, by monotonicity,
\[ \cM(t)  = \cM_V(t) \text{ for all } t<0, \]
where $V$ is an $F$-stationary varifold with $\lambda(V) \leq \Lambda -\eps$. Note that if $V$ is a stationary cone, this holds for all $t \in \RR$ by unit-regularity and $\Lambda-\eps < 2$, and that $V$ is non-flat  since $(\bOh,0) \in\sing \cM_j$ for all $j$.

It remains to prove that  $V \not \in \cS_n^\textrm{gen}$. This follows from \cite{ColdingIlmanenMinicozzi} as we now explain (cf.\ \cite[Theorem 0.2]{ColdingIlmanenMinicozzi}). We now recall \cite[(2.11)]{ColdingIlmanenMinicozzi}. Fix a countable dense subset $\{f_k\}$  of the unit ball in $C^0_c(\RR^{n+1})$ and set 
\[
d_*(\mu_1,\mu_2) : = \sum_k 2^{-k}\left| \int f_k e^{-\tfrac 1 4 |\cdot|^2}d\mu_1 -\int f_k e^{-\tfrac 1 4 |\cdot|^2}d\mu_2 \right|
\]
for $\mu_1,\mu_2$ Radon measures with $F(\mu_i) < \infty$. (This is denoted $d_V$ in \cite{ColdingIlmanenMinicozzi}.) One can check that $d_*$ metrizes the weak-$*$ topology on the space of Radon measures $\mu$ with $F(\mu)<\infty$. By \cite[Corollary 2.12]{ColdingIlmanenMinicozzi} there is $c_0$ so that if $V$ is an $F$-stationary varifold with $F(V) \leq 2-\eps$ and 
\begin{equation}\label{eq:CIM-rigidity}
d_*(V,\cS_n^\textrm{gen})\leq 2c_0 \qquad \implies \qquad V\in \cS_n^\textrm{gen},
\end{equation}
where the $d_*$-metric on a varifold $V$ is the metric with respect to the induced Radon measure $\Vert V \Vert$. 

Returning to the above setup, assume, for contradiction, that $V  \in \cS_n^\textrm{gen}$. For each $j$, let $\tilde\cM_j$ denote the rescaled Brakke flow around $(\bOh,0)$ at scale $1$ corresponding to $\cM_j$. We will use a few times that, by Remark \ref{rema:monotonicity-nearly-constant-contr-arg}, we have for any sequence $\tau_j \in [0,\infty)$, 
\begin{equation} \label{eq:limit-cF-flows-aux}
	\tilde\cM_j(\cdot+\tau_j) \rightharpoonup V',
\end{equation}
where $V'$ is a $F$-stationary $n$-varifold. By assumption, $V'= V$ if $\tau_j\equiv 0$.
Since $(\bOh,0) \not \in \sing_\textrm{gen}\cM_j$ for each $j$, there is a sequence $\tau_{k,j}'\to\infty$ so that \eqref{eq:limit-cF-flows-aux} gives
\[ \tilde\cM_j(\cdot+\tau_{k,j}') \rightharpoonup V_j'\not \in \cS_n^\textrm{gen} \text{ as } k \to \infty, \]
so, by \eqref{eq:CIM-rigidity}, 
\begin{equation} \label{eq:limit-cF-flows-far}
	d_*(\tilde\cM_j(\cdot+\tau_{k,j}'),\cS_n^\textrm{gen}) \geq 2c_0-o(1) \text{ as } k\to\infty.
\end{equation}
Take $\bar \tau_j \in [0, \infty]$ to be as large as possible so that
\[ d_*(\tilde\cM_j(\tau),\cS_n^\textrm{gen}) < c_0 \text{ for all } \tau \in [0,\bar \tau_j). \]
By \eqref{eq:limit-cF-flows-far}, $\bar \tau_j < \infty$ for all $j$. Moreover, \eqref{eq:limit-cF-flows-aux} with $\tau_j \equiv 0$ gives that $\bar \tau_j\to\infty$ as $j\to\infty$. Up to a subsequence, we can assume from \eqref{eq:limit-cF-flows-aux} that $\tilde\cM_j(\cdot+\bar\tau_j) \rightharpoonup \bar V$, where $\bar V$ is a $F$-stationary $n$-varifold. Then, by our choice of $\bar \tau_j$, 
\[
d_*(\bar V,\cS_n^\textrm{gen}) \leq d_*(\tilde\cM_j(\bar \tau_j-1),\cS_n^\textrm{gen}) + o(1) \leq c_0 + o(1),
\]
so $\bar V\in \cS_n^\textrm{gen}$ by \eqref{eq:CIM-rigidity}. However, this yields
\[
c_0 \leq \sup_{\tau \in [0,1]} d_*(\tilde\cM_j(\bar \tau_j+\tau),\cS_n^\textrm{gen}) \leq \sup_{\tau \in [0,1]} d_*(\tilde\cM_j(\bar \tau_j+\tau),\bar V) = o(1).
\]
(In the first inequality we used the definition of $\bar \tau_j$. In the second, we used that $\bar V \in \cS_n^\textrm{gen}$. In the third, we used that $\tilde\cM_j(\cdot+\bar\tau_j) \rightharpoonup \bar V$ as $j \to \infty$.) This is a contradiction to $c_0 > 0$, completing the proof. 
\end{proof}

\begin{lemma}\label{lemm:geo-arg-2}
Assume that $(\diamondsuit_{n,\Lambda})$ holds and take any $((\bOh,0),\cM_j)$, $(X_j',\cM_j') \in \sP(j^{-1})$ such that:
\begin{itemize}
	\item $\cM_j$, $\cM'_j $ don't cross smoothly, 
	\item $\cM_j\rightharpoonup \cM$, 
	\item $\cM_j'\rightharpoonup\cM'$,
	\item $X_j'\to  X'$,
\end{itemize}
and $V$ be as in Lemma \ref{lemm:limit-cF-flows} (applied with $\cM_j$). Then,
\begin{enumerate}
	\item[(a)] $\cM' = \cM_V$ whenever the latter is nonvanishing, and
	\item[(b)] $X' \in \spine V$.
\end{enumerate}
\end{lemma}
\begin{proof}
Denote $X' = (\bx', t')$. 

By Lemma \ref{lemm:limit-cF-flows} (applied to $\cM_j,\cM_j'-X'_j$), there exist two non-flat and non-generic $F$-stationary cyclic integral $n$-varifolds $V$, $V'$ with $\lambda(V), \lambda(V') \leq \Lambda-\eps$ so that
\[ \cM = \cM_V \text{ and } \cM' = \cM_{V'}(\cdot+t') + \bx', \]
whenever the right hand sides are nonvanishing. 

By Lemma \ref{lemm:smooth-cross}, $\cM$ and $\cM'$ do not cross smoothly (otherwise some $\cM_j$, $\cM_j'$ would). Hence, Proposition \ref{prop:geo-input} implies that $V=V'$ and $X'\in\spine V$, as desired.
\end{proof}

\subsection{Graphical distance estimate for nearby rescaled flows}

Going forward we fix $n\in\{2,3,\dots\}$, $\Lambda \in (0,2]$, $\eps > 0$, so that \eqref{heart_cond} and \eqref{diamond_cond} both hold. Recall that constants $\rho_0,r_0>0$ were then chosen in Proposition \ref{prop:uniform-est-spine-eig} (depending only on $n,\Lambda,\eps$).

To understand the next lemma, it may be useful to recall how $\cR_{\rho,r}$ is defined for a rescaled flow in Definition \ref{defi:rescaled-brakke-regularity-scale}. 

\begin{lemma}\label{lemm:rescaled-flows-converge-graphical}
Let $\tilde \cM_j'$, $\tilde \cM_j''$ be sequences of unit-regular cyclic integral rescaled Brakke flows in $\RR^{n+1}$, defined for $\tau \in [T_j,\infty)$ with $T_j\to-\infty$, which satisfy
\begin{itemize}
\item $\lambda(\tilde \cM_j')$, $\lambda(\tilde\cM_j'') \leq \Lambda -\eps$,
\item $\tilde\cM_j'$, $\tilde\cM_j''$ do not cross smoothly,
\item $\tilde\cM_j',\tilde\cM_j'' \rightharpoonup \tilde \cM_V$ for some $F$-stationary cyclic integral $n$-varifold $V$, where $\tilde \cM_V$ denotes the rescaling about $(\bOh, 0)$ (see Definition \ref{defi:rescaled-flow}) of the flow $\cM_V$ associated to $V$ (see Definition \ref{defi:shrinking-flow}).
\end{itemize}
Then, for sufficiently large $j$:
\begin{enumerate}
	\item[(a)] $\cR_{\rho_0,r_0}\tilde \cM_j'(\bOh,0) \neq \emptyset$.
	\item[(b)] There is an increasing sequence of relatively open sets $\tilde W_j$ exhausting $\reg \tilde\cM_V$, with the property that the portions of $\tilde \cM_j'$, $\tilde \cM_j''$ within a  fixed vertical distance from $\tilde W_j$ are normal graphs of $\tilde w_j'$, $\tilde w_j'' \in C^\infty(\tilde W_j)$, and 
	\[ \tilde w_j', \; \tilde w_j''\to 0 \text{ in } C^\infty_\textnormal{loc}(\reg \tilde\cM_V). \]
	Up to switching the unit normal of $\reg \tilde \cM_V$ we can assume that $\tilde w_j' \geq \tilde w_j''$.
	\item[(c)] There is a sequence of second order elliptic operators $L_j$ on $\tilde W_j$ so that
	\[
	( \partial_\tau - L_j) (\tilde w_j'-\tilde w_j'') = 0 \text{ on } \tilde W_j,
	\]
	\[ L_j \to L \text{ in } C^\infty_\textrm{loc}(\reg \tilde \cM_V), \]
	where $L$ is the operator from \eqref{eq:first-eig-L}.
\end{enumerate}
\end{lemma}

\begin{remark} \label{rema:regmv}
	If $\tilde \cM_V$ is as above, then $\reg \tilde \cM_V = \reg V \times \RR$. 
\end{remark}

\begin{proof}[Proof of Lemma \ref{lemm:rescaled-flows-converge-graphical}]
By Proposition \ref{prop:uniform-est-spine-eig}, $\cR_{2\rho_0,r_0/2}\tilde \cM_V(\bOh,0) \neq \emptyset$. This yields the first assertion. The remaining claims follow from unit-regularity and standard geometric arguments (recall that $\reg\tilde\cM_V$ is connected).
\end{proof}

Consider the follow geometric variant of the infimum of vertical distances:

\begin{definition} \label{defi:subset-distance}
	For nonempty subsets $A,B \subset \RR^{n+1}$ we write
	\[
		d(A,B) = \inf\{|\ba-\bb| : \ba \in A,\bb\in B\} .
	\]
	Note that $d(A,B) = 0$ does not mean that $A=B$. 
\end{definition}

\begin{lemma}\label{lemm:unif-Harnack}
There is $H=H(n,\Lambda,\eps)\in [1,\infty) $ with the following property. 

Consider $\tilde \cM'_j,\tilde \cM_j'',\tilde\cM_V,\tilde W_j, \tilde w_j',\tilde w_j''$ as in Lemma \ref{lemm:rescaled-flows-converge-graphical}. Assume that
\[
\tilde d_j : = d(\cR_{\rho_0,r_0}\tilde \cM_j'(\bOh,0),\supp\tilde \cM_j''(0))> 0
\]
for $j=1,2,\dots$, and define
\[
\tilde u_j : = \tilde d_j^{-1}(\tilde w_j' - \tilde w_j'').
\]
Then, after perhaps passing to a subsequence:
\begin{enumerate}
\item[(a)] $\tilde u_j \to \tilde u$ in $C^\infty_\textnormal{loc}(\reg V \times (-\infty,0))$ with $\partial_\tau \tilde u = L\tilde u$ and $\tilde u \geq 0$.
\item[(b)] At rescaled time $\tau = -1$,
\[
\sup_{\cR_{\rho_0,r_0}V (\bOh)} \tilde u(\cdot,-1) \leq H.
\]
\item[(c)] For each rescaled time $\tau_0 <-1$,
\[
\sup_{\cR_{3\rho_0,r_0/3}V(\bOh)} \tilde u(\cdot,\tau_0) \leq H \inf_{\cR_{\rho_0,r_0}V (\bOh)} \tilde u(\cdot ,\tau_0+1). 
\]
\end{enumerate}
\end{lemma}
\begin{proof}
By Lemma \ref{lemm:rescaled-flows-converge-graphical}, the functions $\tilde u_j$ satisfy a second order parabolic PDE 
\[
(\partial_t - L_j) \tilde u_j = 0 \text{ on } \tilde W_j \to \reg \tilde \cM_V, 
\]
\[ L_j \to L \text{ in } C^\infty_\textrm{loc}(\reg\tilde\cM_V). \]
Recall, per Remark \ref{rema:regmv}, that $\reg \tilde \cM_V = \reg V \times \RR$. 
By a standard geometric argument relating graphical separation to distance we find
\begin{equation}\label{eq:graph-dist-relate}
\inf_{\cR_{\rho_0/2,2r_0}V(\bOh)} \tilde u_j(\cdot,0) \leq 1+o(1). 
\end{equation}
We can now apply parabolic Harnack and Schauder theory in a standard way to conclude (a). The proofs of (b) and (c) will follow similarly except we must argue that the constant $H$ can be chosen to only depend on $n,\Lambda,\eps$. 

We first consider (b). Suppose for contradiction that for each $k=1,2,\dots,$ there is a sequence $\tilde \cM'_{j,k},\tilde \cM_{j,k}'',\tilde\cM_{V_k},\tilde W_{j,k}, \tilde w_{j,k}',\tilde w_{j,k}''$, with $j=1,2,\dots$, as in Lemma \ref{lemm:rescaled-flows-converge-graphical} so that 
\[
\tilde u_{j,k} = \tilde d_{j,k}^{-1}(\tilde w_{j,k}'-\tilde w_{j,k}'')
\]
converges as $j\to\infty$ in $C^\infty_\textrm{loc}(\reg V_k\times (-\infty,0))$ to $\tilde u_k$ with $\partial_\tau \tilde u_k = L_{V_k}\tilde u_k$ but so that there are $\bx_k \in \cR_{\rho_0,r_0}V_k(\bOh)$ with 
\begin{equation}\label{eq:unif-Harnack-unif-contra2}
 \tilde u_k(\bx_k,-1) \geq 2k.
\end{equation}
We can pass to a subsequence so that $V_k\rightharpoonup V$. Choose $j=j(k)$ so that 
\begin{equation}\label{eq:unif-Harnack-unif-contra2-BIG}
\tilde u_{j(k),k}(\bx_k,-1) \geq k.
\end{equation}
Using \eqref{eq:graph-dist-relate} we also choose $\by_k \in \cR_{\rho_0/2,2r_0}V_k$ with 
\begin{equation}\label{eq:unif-Harnack-unif-contra2-SMALL}
\tilde u_{j(k),k}(\by_k,0) \leq 2.
\end{equation}
Up to passing to a subsequence, $\bx_k\to \bx,\by_k\to\by \in \reg V$. Writing $L$ for the $L$-operator on $\reg V$, we can choose a finite chain of overlapping coordinate balls (from $\bx$ to $\by$ in $\reg V$ so that in each ball, the $L$-operator has controlled ellipticity and coefficients. Since the $L_j$ operator is smoothly converging to the $L$ operator on $\reg \tilde \cM_V$, the same thing holds for $L_j$. Now, we can apply parabolic Harnack on the chain of balls in the usual way to bound
\[
 \tilde u_{j(k),k}(\bx_k,-1) \leq \tilde H \tilde u_{j(k),k}(\by_k,0),
\]
where $\tilde H$ depends only on $V$ and the chosen chain of balls, but not on $k$. However, this contradicts \eqref{eq:unif-Harnack-unif-contra2-BIG} and \eqref{eq:unif-Harnack-unif-contra2-SMALL}. This proves (b). 

Finally, to prove (c) we argue in the same way, except we replace \eqref{eq:unif-Harnack-unif-contra2} by 
\[
\tilde u_k(\bx_k,\tau_k) \geq (2k) \tilde u_k(\by_k,\tau_k + 1).
\]
for some $\tau_k<-1$. The argument is the same as before except we shift rescaled time by $+\tau_k$ to account for the possibility that $\tau_k\to-\infty$. Since the rescaled flow of a shrinker is stationary, this does not change the above analysis. This completes the proof.
\end{proof}

\subsection{Separation and packing estimates}

We first need an elementary growth lemma about nonnegative solutions of $(\partial_\tau - L) \tilde u = 0$ on compact subdomains of shrinkers, which we will then ``de-linearize.'' Below is a parabolic analogue of the linear estimates obtained in \cite{Simon:asymptotic-decay} and \cite[Appendix A]{Wang:smoothing}. 

\begin{lemma} \label{lemm:lin-growth}
Let $\Sigma\subset \RR^{n+1}$ be a compact self-shrinker with (possibly empty) boundary. Let $\varphi>0$ denote the first eigenfunction of the $L$-operator on $\Sigma$ with Dirichlet boundary conditions. Suppose that $\tilde u \in C^\infty(\Sigma\times [\tau_0,\tau_1])$ satisfies
\[ \partial_\tau \tilde  u = L\tilde u, \; \tilde u\geq 0. \]
Then, for all $\tau \in [\tau_0, \tau_1]$ we have
\[
\int_\Sigma \tilde u(\bx,\tau) \varphi(\bx) e^{-\tfrac 14 |\bx|^2} \, d\bx \geq e^{-\mu(\Sigma) \cdot (\tau-\tau_0)} \int_\Sigma \tilde u(\bx,\tau_0) \varphi(\bx) e^{-\tfrac 14 |\bx|^2} \, d\bx.
\]
\end{lemma}
\begin{proof}
Set 
\[
\tilde U(\tau) : = \int_\Sigma \tilde u(\bx,\tau) \varphi(\bx) e^{-\tfrac 14 |\bx|^2} \, d\bx. 
\]
Differentiating under the integral sign and then integrating by parts we find
\begin{align*}
\tilde U'(\tau) & = \int_\Sigma (L\tilde u(\bx,\tau)) \varphi(\bx) e^{-\tfrac 1 4|\bx|^2} \\
& = \int_\Sigma \tilde  u(\bx,\tau) (L \varphi(\bx)) e^{-\tfrac 1 4|\bx|^2} - \int_{\partial\Sigma} \tilde  u(\bx,\tau) (\partial_\zeta \varphi(\bx)) e^{-\tfrac 1 4|\bx|^2} \\
& \geq - \mu(\Sigma) \int_\Sigma  \tilde  u(\bx,\tau) \varphi(\bx) e^{-\tfrac 1 4|\bx|^2} = -\mu(\Sigma) \tilde U(\tau).
\end{align*}
Above, $\zeta$ is the outward pointing unit conormal and we used that $\partial_\zeta \varphi \leq 0$ due to the Dirichlet boundary conditions and $\varphi \geq 0$. Integrating this proves the assertion. 
\end{proof}

Now continue to fix $n\in\{2,3,\dots,\}$, $\Lambda \in (0,2]$, $\eps>0$ so that \eqref{heart_cond} and \eqref{diamond_cond} hold. 

\begin{remark}\label{rema:constants}
We have chosen several other constants depending only on $n,\Lambda,\eps$, as we recall here:
\begin{enumerate}
	\item The covering constant $C$ was fixed in Lemma \ref{lemm:covering}. 
	\item The gain of decay constant $\kappa$ and regularity scale constants $\rho_0,r_0$ were fixed in Proposition \ref{prop:uniform-est-spine-eig}. 
	\item The Harnack constant $H$ was fixed in Lemma \ref{lemm:unif-Harnack}. 
\end{enumerate}
We now fix the scale size constant $\gamma \in (0,1)$ by 
\begin{equation}\label{eq:choice-gamma}
 \gamma =  \min\left\{\left(\tfrac{1}{\sqrt{5}} H^{-2} (\tfrac{5}{4e})^{2\kappa + n+2}\right)^{1/\kappa}, (2C)^{-1/\kappa}, \tfrac{\sqrt{3}}{2e}\right\}
\end{equation}
Note that $\gamma=\gamma(n,\Lambda,\eps) < \tfrac 12$. 
\end{remark}

With these preparations, we can now discuss the key covering/separation result. To state it, it is convenient to consider the projection operator:
\[ \pi : (X,\cM) \mapsto X. \]

\begin{proposition}\label{prop:main-cover-sep}
There exists $\eta_0 \in (0,1)$ with the following property. 

Let $\cQ \subset \sP(\eta_0)$ be such that:
\begin{itemize}
\item $(\bOh,0) \in \pi(\cQ)$, and
\item if $(X',\cM')$ and $(X'',\cM'') \in  \cQ$ then $\cM'$ and $\cM''$ don't cross smoothly. 
\end{itemize}
Then there exist $D \in \NN$ with $D \leq n+2$ and points $X_1=(\bx_1,t_1), \ldots, X_K=(\bx_K,t_K) \in \pi(\cQ)$ such that:
\begin{enumerate}
\item[(a)] $\pi(\cQ) \subset \cup_{i=1}^K P(X_i,\gamma)$ with $K \leq C \gamma^{-D}$;
\item[(b)] for $(X',\cM') \in \cQ \cap \pi^{-1}(P(X_i,\gamma))$ it holds that
\begin{align*}
	\cR_{2\rho_0,2r_0}\cM'(\bOh,-4)   & \neq \emptyset, \\
	\cR_{2\gamma\rho_0,2\gamma r_0}\cM'(\bx_i,t_i-4\gamma^2)  & \neq \emptyset; \text{ and,}
\end{align*}
\item[(c)] for $(X',\cM'), (X'',\cM'') \in \cQ \cap \pi^{-1}(P(X_i,\gamma))$ we have:
\begin{align*}
& \gamma^{- (\kappa+D)} \cdot 2^{-1} d(\cR_{2\rho_0,2r_0}\cM'(\bOh,-4) , \supp\cM''(-4) ) \\
& \qquad \leq (2\gamma)^{-1} d(\cR_{2\gamma\rho_0,2\gamma r_0}\cM'(\bx_i,t_i-4\gamma^2)  , \supp\cM''(t_i-4\gamma^2)  ). 
\end{align*}
\end{enumerate}
\end{proposition}

\begin{remark}
	One could have replaced $\gamma^{-(\kappa + D)}$ by $\gamma^{-((2-\delta)\kappa + D)}$ with any $\delta \in (0, 2)$ in (c); cf. Proposition \ref{prop:uniform-est-spine-eig} (b).
\end{remark}

\begin{proof}[Proof of Proposition \ref{prop:main-cover-sep}]
Suppose for contradiction that no such $\eta_0$ exists. Then there would exist a sequence $\cQ_j \subset  \sP(j^{-1})$, with $j$ large, so that
\begin{itemize}
\item there is $((\bOh,0),\cM_j) \in \cQ_j$ and
\item if $(X',\cM'),(X'',\cM'') \in  \cQ_j$ then $\cM'$ and $\cM''$ don't cross smoothly,
\end{itemize}
but that there is no $D \leq n+2$ and a collection of points in $\pi(\cQ_j)$ so that (a), (b), and (c) hold.

Pass to a subsequence so that $\cM_j \rightharpoonup \cM$. By Lemma \ref{lemm:limit-cF-flows}, there exists a non-planar and non-generic $F$-stationary cyclic integral $n$-varifold $V$ with $\lambda(V) \leq \Lambda - \eps$ and 
\[ \cM = \cM_V \]
whenever the right hand side is nonvanishing. Set $\Pi = \spine V$ and $D=D(\Pi)$. Note that $D \leq n+2$ by Definition \ref{defi:spacetime-spine}.

For $(X_j',\cM_j') \in \cQ_j $, we can pass to a subsequence so that $X_j'\to X'$ and $\cM_j'\rightharpoonup \cM'$. Then, by Lemma \ref{lemm:geo-arg-2}, $X' \in \Pi$ and
\[ \cM'=\cM = \cM_V \]
whenever the right hand side is nonvanishing. After discarding finitely many $j$, we can assume that
\[ \pi(\cQ_j) \subset U_\gamma(\Pi) \cap P((\bOh,0),1). \]
By Lemma \ref{lemm:covering} we can thus find $X_{j,1},\dots,X_{j,K_j} \in \pi(\cQ_j)$ so that
\[ \pi(\cQ_j) \subset \cup_{i=1}^{K_j}P(X_{j,i},\gamma) \text{ and } K_j \leq C\gamma^{-D}. \]
Thus, (a) holds for this collection of points. 

Passing to a further subsequence (not relabeled) if necessary, we can assume that either (b) or (c) fails for every $j$ and some $i=i(j)$. We will then derive a contradiction. For notational simplicity, we will write $X_j=(\bx_j,t_j)$ instead of $X_{j,i}$. 

We start with (b). Take $(X'_j,\cM'_j) \in \cQ_j \cap \pi^{-1}(P(X_{j},\gamma))$. We claim that, for all large $j$,
\[
\cR_{2\rho_0,2r_0}\cM_j'(\bOh,-4) \neq \emptyset \text{ and } \cR_{2\gamma\rho_0,2\gamma r_0}\cM_j'( \bx_j,t_j-4\gamma^2)  \neq \emptyset.
\]
We just prove the latter statement; the former follows similarly. Let $\tilde \cM_j'$ denote the rescaled Brakke flow corresponding to $\cM_j'$ centered at $X_{j}=(\bx_j,t_j)$ at scale $2\gamma$ (see Definition \ref{defi:rescaled-flow}). In other words, we set
\[
\tilde \cM_j'(\tau) = (2\gamma)^{-1}e^{\tau/2} (\cM_j'(t_{j}-4\gamma^2 e^{-\tau}) - \bx_{j}). 
\]
Since $\gamma$ is fixed, we find that $\tilde\cM_j'\rightharpoonup \tilde\cM_V$. Lemma \ref{lemm:rescaled-flows-converge-graphical} implies that, for $j$ large,
\[
\cR_{\rho_0,r_0}\tilde\cM_j'(\bOh,0)  \neq 0.
\]
This yields the desired nonemptiness after returning to the unrescaled flow (see Remark \ref{rema:reg-scale}). This is a contradiction to our assumption that (b) failed for all $j$. 

We can thus assume (c) fails for all $j$. By our setup, there exist
\[
(X_j',\cM_j'), (X_j'',\cM_j'') \in \cQ_j \cap \pi^{-1}(P(X_{j},\gamma))
\]
so that
\begin{multline}\label{eq:failure-separation-prelim}
(2\gamma)^{-1} d(\cR_{2\gamma\rho_0,2\gamma r_0}\cM_j'( \bx_j,t_{j}-4\gamma^2) , \supp  \cM_j''(t_{j}-4\gamma^2))\\
< \gamma^{-(\kappa + D)} \cdot 2^{-1}d(\cR_{2\rho_0,2r_0}\cM_j'(\bOh,-4) ,\supp \cM_j''(-4)). 
\end{multline}
We now derive a contradiction to \eqref{eq:failure-separation-prelim} (this will complete the proof of the proposition). It is important to shift the spatial center on the right-hand side of \eqref{eq:failure-separation-prelim}. Below, recall that $r_0 > 2$ by Proposition \ref{prop:uniform-est-spine-eig} and $|\bx_j| < 1$ by Definition \ref{defi:Peta0}, so $2r_0 - |\bx_j| > 0$. We note that
\[
\cR_{2\rho_0,2r_0 - |\bx_j|} \cM_j'(\bx_j,-4) \subset \cR_{2\rho_0,2r_0}\cM_j'(\bOh,-4) 
\]
by Remark \ref{rema:reg-scale}. Thus, we have\footnote{Proposition \ref{prop:uniform-est-spine-eig} gives that $\cR_{3\rho_0,r_0/3}V(\bOh) \neq \emptyset$; this ensures that the sets on the right-hand-side of \eqref{eq:failure-separation}, \eqref{eq:failure-separation-rescaled-1}, \eqref{eq:failure-separation-rescaled-2} are nonempty for large $j$.}
\begin{multline}\label{eq:failure-separation}
(2\gamma)^{-1} d(\cR_{2\gamma\rho_0,2\gamma r_0}\cM_j'( \bx_j,t_{j}-4\gamma^2) , \supp \cM_j''(t_{j}-4\gamma^2))\\
< \gamma^{-(\kappa + D)} \cdot 2^{-1}d(\cR_{2\rho_0,2r_0-|\bx_j|}\cM_j'(\bx_j,-4) ,\supp\cM_j''(-4)). 
\end{multline}
Let's translate \eqref{eq:failure-separation} into a statement about  the rescaled flows $\tilde\cM_j',\tilde\cM_j''$ around $X_j$ at scale $2\gamma$. 
Note that under this rescaling, non-rescaled time turns into rescaled time as:
\[ t = t_j-4\gamma^2 \mapsto \tau=0, \]
\[ t=-4 \mapsto \tau = \tau_j : =  2\log \gamma - \log (1+t_{j}/4). \]
In particular, using $|t_j| \leq 1$, the dilation factor in the rescaled flow at $\tau_j$  satisfies 
\[
(2\gamma)^{-1} e^{\tau_j/2} =  \tfrac 12 (1+t_j/4)^{-\frac 12} : = \beta_j \in (\tfrac{1}{\sqrt{5}},\tfrac{1}{\sqrt{3}}). 
\]
As such, \eqref{eq:failure-separation} rescales  (see Remark \ref{rema:reg-scale}) to yield
\begin{multline}\label{eq:failure-separation-rescaled-1}
d(\cR_{\rho_0,r_0}\tilde \cM_j'(\bOh,0) , \supp\tilde \cM_j''(0) )\\
< \gamma^{-(\kappa + D)}  \cdot (2 \beta_j)^{-1}   d(\cR_{{2}\beta_j \rho_0,{2}\beta_j(r_0  - |\bx_j|/2)}\tilde \cM_j'(\bOh,\tau_j)  ,\supp\tilde \cM_j''(\tau_j)).
\end{multline}
We now observe that \eqref{eq:failure-separation-rescaled-1} implies:
\begin{multline}\label{eq:failure-separation-rescaled-2}
d(\cR_{\rho_0,r_0}\tilde \cM_j'(\bOh,0) , \supp\tilde \cM_j''(0) )\\
< \gamma^{-(\kappa + D)}  \cdot \tfrac{\sqrt{5}}{2} \cdot d(\cR_{2 \rho_0,r_0/2}\tilde \cM_j'(\bOh,\tau_j)  ,\supp\tilde \cM_j''(\tau_j)).
\end{multline}
Indeed, first note $\beta_j > \tfrac{1}{\sqrt{5}}$ obviously implies the scaling factor bound
\[ (2 \beta_j)^{-1} < \tfrac{\sqrt{5}}{2}. \]
Moreover, we will also have
\[
\cR_{2\rho_0,r_0/2} \tilde\cM_j'(\bOh,\tau_j) \subset \cR_{2 \beta_j \rho_0, 2\beta_j(r_0  - |\bx_j|/2)}\tilde \cM_j'(\bOh,\tau_j) .
\]
by Remark \ref{rema:reg-scale} because $\beta_j \in (\tfrac{1}{\sqrt{5}}, \tfrac{1}{\sqrt{3}})$ implies both
\[
	2 \beta_j \rho_0 < \tfrac{2}{\sqrt{3}} \rho_0 < 2 \rho_0,
\]
and, also using $r_0 > 2$, $|\bx_j| \leq 1$, that
\[
	2 \beta_j(r_0 - |\bx_j|/2) > 2 \tfrac{1}{\sqrt{5}} \cdot \tfrac{3}{4} r_0 = \tfrac{3}{2\sqrt{5}} r_0 > r_0/2.
\]
This completes the proof that \eqref{eq:failure-separation-rescaled-1} implies \eqref{eq:failure-separation-rescaled-2}.

Denote
\[ \tilde d_j : = d(\cR_{\rho_0,r_0}\tilde \cM_j'(\bOh,0) , \supp\tilde \cM_j''(0) ). \]
Observe that \eqref{eq:failure-separation-rescaled-2} and {the avoidance principle} imply $\tilde d_j >0$.  

We pass to a further subsequence so that $\tau_j \to \tau_0$ with $\tau_0 - 2\log \gamma \in [\log \tfrac 45,\log\tfrac 43]$. Note that \eqref{eq:choice-gamma} guarantees that $\tau_0\leq -2$. Note also that
\[ \tilde \cM_j',\tilde\cM_j'' \rightharpoonup \tilde\cM_V \]
satisfy the hypotheses of Lemma \ref{lemm:rescaled-flows-converge-graphical}. Let $\tilde W_j$, $\tilde w_j'$, $\tilde w_j''$ be as in that lemma with $\tilde w_j' \geq \tilde w_j''$. 
Apply Lemma \ref{lemm:unif-Harnack} to pass
\[ \tilde u_j : = \tilde d_j^{-1}(\tilde w_j'-\tilde w_j'') \]
to a $C^\infty_\textrm{loc}(\reg V\times (-\infty,0))$ limit that satisfies 
\begin{equation}\label{eq:first-appl-harnack}
 \sup_{\cR_{\rho_0,r_0}V (\bOh)} \tilde u(\cdot,-1) \leq H. 
\end{equation}
and
\begin{equation}\label{eq:second-appl-harnack}
\sup_{\cR_{3 \rho_0,r_0/3}V(\bOh)} \tilde u(\cdot,\tau_0) \leq H \inf_{\cR_{\rho_0,r_0}V (\bOh)} \tilde u(\cdot,\tau_0+1). 
\end{equation}
Finally, we note that using $\theta \in (2, 3)$, \eqref{eq:failure-separation-rescaled-2} passes to the limit to yield
\begin{equation}\label{eq:lower-bd-tau0+1}
\inf_{\cR_{3 \rho_0,r_0/3}V(\bOh)} \tilde u(\cdot,\tau_0) \geq  \tfrac{2}{\sqrt{5}} \gamma^{\kappa+D}. 
\end{equation}
Combining \eqref{eq:second-appl-harnack}  and \eqref{eq:lower-bd-tau0+1} we find
\begin{equation}\label{eq:lower-bd-tau0+1-final}
 \tfrac{2}{\sqrt{5}}  H^{-1} \gamma^{\kappa+D} \leq \inf_{\cR_{\rho_0,r_0}V(\bOh)} \tilde u(\cdot,\tau_0+1). 
\end{equation}
Fix a relatively open set $\Omega \subset \reg V$ with smooth boundary such that
\[ \cR_{2\rho_0,r_0/2}V(\bOh)  \subset \Omega \subset \cR_{\rho_0,r_0}V (\bOh). \]
By Proposition \ref{prop:uniform-est-spine-eig},
\[
2 \mu(\Omega)  < -2 \kappa - D : = 2 \mu_0.
\]
We now apply Lemma \ref{lemm:lin-growth} to $\Sigma = \overline \Omega$ on the interval $[\tau_0+1,-1]$ to find 
\begin{equation}\label{eq:appl-growth-estimate}
\int_\Omega \tilde u(\cdot,-1) \varphi e^{-\tfrac 1 4 |\cdot|^2} \geq e^{\mu_0(\tau_0+2)} \int_\Omega \tilde u(\cdot,\tau_0+1) \varphi e^{-\tfrac 1 4 |\cdot|^2},
\end{equation}
where $\varphi>0$ is the first Dirichlet eigenfunction of $L$ on $\Omega$ and $\int_\Omega \varphi e^{-\tfrac 14|\cdot|^2} = 1$. We then claim that we have the following chain of inequalities:
\begin{align*}
H & \geq  \sup_{\cR_{\rho_0,r_0}(\bOh)}\tilde u(\cdot,-1) \\
&\geq \int_\Omega \tilde u(\cdot,-1) \varphi e^{-\tfrac 1 4 |\cdot|^2} \\
& \geq e^{\mu_0(\tau_0+2)} \int_\Omega \tilde u(\cdot,\tau_0+1) \varphi e^{-\tfrac 1 4 |\cdot|^2}\\
& \geq e^{\mu_0(\tau_0+2)} \inf_{\cR_{\rho_0,r_0}V(\bOh)} \tilde u(\cdot,\tau_0+1)\\
& \geq\tfrac{2}{\sqrt{5}}  H^{-1} \gamma^{\kappa + D}  e^{\mu_0(\tau_0+2)} \\
& \geq\tfrac{2}{\sqrt{5}}  H^{-1} \gamma^{\kappa + D + 2\mu_0} (\tfrac{4e}{5})^{2\mu_0}.
\end{align*}
The first inequality follows from \eqref{eq:first-appl-harnack}, the second and fourth from the normalization of $\varphi$ and $\Omega \subset \cR_{\rho_0,r_0}(\bOh)$, the third from \eqref{eq:appl-growth-estimate}, the second to last follows from \eqref{eq:lower-bd-tau0+1-final} and the final inequality from $e^{\tau_0} \geq \tfrac 4 5 \gamma^2 $.

Noting that $D\leq n+2$ (by Definition \ref{defi:spacetime-spine}), we have $2\mu_0 \geq - 2\kappa - n-2$. Thus, we can rearrange the previous expression to read
\[
\tfrac{2}{\sqrt{5}}(\tfrac{5}{4e})^{2\kappa + n+2} \leq  H^2 \gamma^{\kappa}.
\]
This contradicts \eqref{eq:choice-gamma}, completing the proof. 
\end{proof}

\section{Density drop}\label{sec:density-drop}

\subsection{Setup} 

We continue to fix $n\in\{2,3,\dots,\}$, $\Lambda \in (0,2]$, $\eps>0$ so that \eqref{heart_cond} and \eqref{diamond_cond} hold. Recall the definition of the constants $C,\kappa,\rho_0,H$ (cf.\ Remark \ref{rema:constants}). We have also fixed $\gamma$ in \eqref{eq:choice-gamma} and $\eta_0$ in Proposition \ref{prop:main-cover-sep}. The constants $C,\kappa,\rho_0,H,\gamma,\eta_0$ all depend only on the values of $n,\Lambda,\eps$. 

In this section, we consider a closed interval $0 \in S$  with $|S|>0$  and a local foliation $\{M_s\}_{s\in S}$ with $M_s\subset \RR^{n+1}$ closed connected embedded hypersurfaces having 
\begin{equation}\label{eq:S-foliation-speed-assumpt}
	d(M_s, M_{s'}) \geq \tfrac12 |s-s'|
\end{equation} 
(e.g., a foliation with unit normal speed and small $|S|$ will do) and
\begin{equation}\label{eq:S-entropy-bd-assumpt}
\lambda(M_s) \leq \Lambda -\eps.
\end{equation} 
Shrinking $S$, we can assume that the $M_s$ are uniformly smooth and embedded, so in particular there is $t_0>0$ so that if $(\bx,t) \in \sing \cM,\cM \in \fF(M_s), s\in S$ then $t\geq t_0$. 

\subsection{Singular points on foliation Brakke flows}

For any $Q \subset S$ we will write
\begin{align}\label{eq:singS-ng-g}
\sing(Q) & : = \cup_{s \in Q} \cup_{\cM \in \fF(M_s)} \sing \cM,\\
\sing_\textrm{gen}(Q) & : = \cup_{s \in Q} \cup_{\cM \in \fF(M_s)} \sing_\textrm{gen} \cM \nonumber \\
\sing_\textrm{non-gen}(Q) & : = \cup_{s \in Q} \cup_{\cM \in \fF(M_s)} \sing_\textrm{non-gen} \cM. \nonumber
\end{align}
By Corollary \ref{coro:dont-cross-dens-well-defined} we see that if
\[ \cM,\cM' \in \cup_{s \in S} \fF(M_s) \text{ and } X=(\bx,t) \in \supp\cM\cap\supp\cM' \text{ has } t > 0, \]
then:
\begin{itemize}
	\item $\cM$, $\cM' \in \fF(M_{s_0})$ for $s_0\in S$,
	\item the set of tangent flows to $\cM$, $\cM'$ at $X$ agree, 
	\item as do the densities $\Theta_{\cM}(X)$, $\Theta_{\cM'}(X)$. 
\end{itemize}
In particular, just like how, for each $\cM \in \cup_{s \in S} \fF(M_s)$,
\[ \sing \cM = \sing_\textrm{gen} \cM \cup \sing_{\textrm{non-gen}} \cM \]
is a disjoint union (by \cite{ColdingIlmanenMinicozzi,ColdingMinicozzi:uniqueness-tangent-flow}), so is
\[ \sing(S) = \sing_\textrm{gen}(S) \cup \sing_{\textrm{non-gen}}(S). \]
Moreover, for any $X \in \sing(S) \iff X \in \sing \cM$, $\cM \in \fF(M_s)$, $s \in S$, we can define
\[ \Theta(X) := \Theta_\cM(X) \]
unambiguously without reference to $\cM$.

\subsection{The improved quantity}

Let $\fN \subset \sing_\textrm{non-gen}(S)$ be a closed subset with respect to the parabolic topology and define,
\[ \cD_\fN : S \to \{ -\infty \} \cup [1, \infty), \]
\[
\cD_\fN(s):= \sup \{ \Theta(X) : X \in \fN \cap \sing \cM, \; \cM \in \sF(M_s) \},
\]
where we use the standard convention $\sup\emptyset=-\infty$. 

\begin{lemma}\label{lemm:usc-DfS}
The function $s\mapsto \cD_\fN(s)$ is upper semi-continuous, i.e., for all $s_0 \in S$,
\[
\limsup_{s\to s_0} \cD_\fN(s) \leq \cD_\fN(s_0).
\]
\end{lemma}
\begin{proof}
Choose $s_i\to s_0$, $s_i \neq s_0$, $\cM_i \in \fF(M_{s_i})$, $X_i \in \sing\cM_i \cap \fN$ with 
\[
\lim_i \Theta(X_i) = \limsup_{s\to s_0} \cD_\fN(s). 
\]
Passing to a subsequence, $\cM_i\rightharpoonup \cM \in \fF(M_{s_0})$ and $X_i\to X$. Since $\fN$ is closed, $X\in\fN$. Thus,
\[
\cD_\fN(s_0) \geq \Theta(X) \geq \lim_i \Theta(X_i) = \limsup_{s\to s_0} \cD_\fN(s)
\]
by upper semicontinuity of density. This completes the proof. 
\end{proof}

The main result of this section is as follows. Below, $\eta_0$ is given by Proposition \ref{prop:main-cover-sep}.

\begin{proposition}\label{prop:dens-fS-drop}
We have
\begin{equation*}
\liminf_{s\to s_0} \cD_\fN(s) \leq \cD_\fN(s_0) - \eta_0 
\end{equation*}
for any $s_0 \in S$. 
\end{proposition}

\begin{corollary}\label{coro:dense-no-fS}
There is a relatively open dense subset $S' \subset S$ with $\sing(S') \cap \fN =\emptyset$. 
\end{corollary} 
\begin{proof}
Note that if $\cD_\fN(s) < 1$ then $\cD_\fN(s) = -\infty$ by integrality. Thus
\[ S' : = \cD_\fN^{-1}(-\infty) = \cD_\fN^{-1}([-\infty,1)) \]
is open using the upper-semicontinuity from Lemma \ref{lemm:usc-DfS}. On the other hand, iterating Proposition \ref{prop:dens-fS-drop} finitely many times, we see that $S'$ is dense. This completes the proof. 
\end{proof}

The remainder of this section is devoted to the proof of Proposition \ref{prop:dens-fS-drop}. Note that the failure of Proposition \ref{prop:dens-fS-drop} for some $s_0 \in S$ would mean that
\begin{equation}\label{eq:dens-fS-drop-fails}
\liminf_{s\to s_0} \cD_\fN(s) > \cD_\fN(s_0) - \eta_0 . 
\end{equation}
As such, by combining  \eqref{eq:dens-fS-drop-fails} with Lemma \ref{lemm:usc-DfS}, we see that if Proposition \ref{prop:dens-fS-drop} was false, we can shrink $S$ while preserving  $s_0 \in S$ so that
\begin{equation}\label{eq:dens-fS-not-much-change}
\cD_\fN(s_0) - \eta_0 \leq \cD_\fN(s) \leq \cD_\fN(s_0) + \eta_0 \text{ for all } s \in S.
\end{equation}
For the sake of contradiction, we can thus assume that \eqref{eq:dens-fS-not-much-change} holds in the remainder of the proof. 
\begin{lemma}\label{lemm:unif-dens-drop-lb}
There is $r_0\in (0,\tfrac 12 \sqrt{t_0}) $ so that 
\[
\Theta_\cM(X,r_0) \leq \cD_\fN(s) + 3\eta_0
\]
for every $s \in S$, $\cM\in \fF(M_s)$, and $X \in  \fN \cap \sing \cM$.
\end{lemma}
\begin{proof}
If not, there exist $X_i$, $\cM_i$, $s_i$ as above, and $r_i\to 0$ so that 
\[
\Theta_{\cM_i}(X_i,r_i) > \cD_\fN(s_i) + 3\eta_0. 
\]
Passing to a subsequence (not relabeled), let $s_i \to s_\infty \in S$. Using \eqref{eq:dens-fS-not-much-change} twice we get 
\[
\cD_\fN(s_i) \geq \cD_\fN(s_0) -\eta_0 \geq \cD_\fN(s_\infty) - 2\eta_0. 
\] 
Thus, we find that 
\[
\Theta_{\cM_i}(X_i,r_i) > \cD_\fN(s_\infty) + \eta_0. 
\]
For any $\rho \in (0,\sqrt{t_0})$, monotonicity yields
\[
\Theta_{\cM_i}(X_i,\rho) > \cD_\fN(s_\infty) + \eta_0. 
\]
for large $i$. Pass to a further subsequence (not relabeled) so that $\cM_i\rightharpoonup \cM_\infty \in \fF(M_{s_\infty})$ and $X_i\to X_\infty$. Then,
\[
\Theta_{\cM_\infty}(X_\infty,\rho) \geq \cD_\fN(s_\infty) + \eta_0. 
\]
and sending $\rho\to 0$, we find 
\begin{equation}\label{eq:gauss-area-drop-small-fS-closed}
\Theta(X_\infty) \geq \cD_\fN(s_\infty) + \eta_0. 
\end{equation}
However, $\fN$ is closed so $X_\infty\in \fN$, contradicting \eqref{eq:gauss-area-drop-small-fS-closed}. This completes the proof. 
\end{proof}

Up to performing a (single) global scaling, we will assume below that $r_0=\eta_0^{-1}$. In particular, by Lemma \ref{lemm:unif-dens-drop-lb} and the definition of $t_0$, we see that 
\begin{equation}\label{eq:no-sing-eta0inv}
\sing(S) \cap \{ X = (\bx, t) : t \in (0,4\eta_0^{-2}]\} = \emptyset.
\end{equation}

\subsection{Covers of high-density points in $\fN$}

We wish to construct appropriate covers of the set of nearly top density points in $\fN$:
\[
\fN_+ : =\{X \in \fN : s \in S, \cM \in \fF(M_s), \Theta(X) \geq \cD_\fN(s) - \eta_0\}.
\]
Let us introduce the notation we will use. For $\ell\in\{0,1,\dots\}$ we will construct a finite set $\cT[\ell]$ together with ``geometric realization maps''
\[ \iota_\ell : \cT[\ell]\to\fN_+. \]
We will also choose ``parent'' maps
\[ \fp : \cT[\ell+1]\to \cT[\ell]. \]
We will write $\fp^{(k)}$ for the iterated parent (grandparent) map, with the usual convention that $\fp^{(0)}=\id : \cT[\ell]\to\cT[\ell]$.

\begin{remark}
Informally, we can think of $\cup_\ell \iota_\ell(\cT[\ell])$ as a tree of points in $\fN_+$. However, it is technically useful to allow multiple elements to map to the same point in $\fN_+$, but allowing them to have different parent elements.
\end{remark}

Now, assuming that we have inductively constructed $\cT[0]\xleftarrow{\fp} \cT[1] \xleftarrow{\fp} \dots \xleftarrow{\fp} \cT[\ell]$, we proceed to define the following for each $\cX \in \cT[\ell]$:
\begin{itemize}
	\item the spacetime neighborhood
	\[
	P_\ell(\cX) : = \bigcap_{j=0}^\ell P(\iota_j(\fp^{(\ell-j)}(\cX)),\gamma^j);
	\]
	\item the restriction of $\fN_+$ to it
	\[
	\fN_{+,\ell}(\cX): = \fN_+ \cap P_\ell(\cX);
	\]
	\item the set of pairs of (point, flow) so that the flow has a high density point lying in $\fN_{+,\ell}(\cX)$, followed by a translation/scaling taking $P(\iota_\ell(\cX),\gamma^\ell)\mapsto P((\bOh,0),1)$
\begin{equation*}
	\begin{split}
	 \cQ_{\ell}(\cX) &: = \{({ \operatorname{ParDil}_{\gamma^{-\ell}}}(X'-\iota_\ell(\cX)),{\operatorname{ParDil}_{\gamma^{-\ell}}}(\cM'-\iota_\ell(\cX))) :\\
	 &\qquad\qquad X' \in \fN_{+,\ell}(\cX) \cap \sing \cM', \cM' \in \fF(M_s), s \in S \},
	\end{split}
\end{equation*}
\end{itemize}
where $\operatorname{ParDil}_\lambda$ is the standard parabolic dilation of a Brakke flow around the space-time origin $(\bOh,0)$ by a factor $\lambda>0$.

Recalling that $\sP(\eta_0)$ is as in Definition \ref{defi:Peta0} and Proposition \ref{prop:main-cover-sep}, and that 
\[ \pi : (X, \cM) \mapsto X, \]
denotes the projection onto the first coordinate in $\sP(\eta_0)$, we have the following result:

\begin{lemma}
If $\cX \in \cT[\ell]$, then:
\begin{enumerate}
	\item[(a)] $\cQ_\ell(\cX)\subset \sP(\eta_0)$, and 
	\item[(b)] the assumptions of Proposition \ref{prop:main-cover-sep} are satisfied with $\cQ_\ell(\cX)$ in place of $\cQ$.
\end{enumerate}
\end{lemma}
\begin{proof}
We first prove (a), i.e., that $\cQ_\ell(\cX)$ satisfies (1)-(5) in Definition \ref{defi:Peta0} (with $\eta=\eta_0$). We first note that since $\fN_{+,\ell} \subset P_\ell(\cX) \subset P(\iota_\ell(\cX), \gamma^\ell)$, we have
\[
{\operatorname{ParDil}_{\gamma^{-\ell}}}(\fN_{+,\ell}(\cX)-\iota_\ell(\cX)) \subset { \operatorname{ParDil}_{\gamma^{-\ell}}}(P_\ell(\cX) - \iota_\ell(\cX)) \subset P((\bOh,0),1),
\]
so (1) holds. Condition (2) follows from $\gamma < 1$ and \eqref{eq:no-sing-eta0inv}. Condition (3) follows from \eqref{eq:S-entropy-bd-assumpt}. Condition (4) follows from that $\fN \subset \sing_\textrm{non-gen}(S)$. Finally, for any $\cM' \in \fF(M_s)$ and $X' \in \fN_{+,\ell}(\cX) \cap \sing \cM'$, we have that
\[ \Theta_{\cM'}(X') \geq \cD_{\fS}(s)-\eta_0 \]
by definition of $\fN_+$ and 
\[
\Theta_{\cM'}(X', \eta_0^{-1}) \leq \cD_\fS(s) + 3\eta_0 
\]
by Lemma \ref{lemm:unif-dens-drop-lb} and the fact that we scaled so that $r_0= \eta_0^{-1}$. Putting this together, we find using monotonicity that
\[
\Theta_{\cM'}(X',\eta_0^{-1}\gamma^\ell) \leq \Theta_{\cM'}(X',\eta_0^{-1}) \leq \Theta_{\cM'}(X') + 4\eta_0 .
\]
Translating and scaling, this yields condition (5). Thus (a) holds.

Now consider (b). The first bullet point follows from how we translated and scaled the pairs, and the second bullet point follows from Lemma \ref{lemm:fF-dontsmoothlycross}. 
\end{proof}

Thus, Proposition \ref{prop:main-cover-sep} applies to each $\cQ_\ell(\cX)$, $\cX \in \cT[\ell]$ and yields the following corollary (after undoing the various definitions of $\cQ_\ell(\cX)$, etc.):

\begin{corollary}\label{coro:appl-covering-separation-depth-ell}
Let $\cX \in \cT[\ell]$. Then, there exist $D_\ell(\cX) \in \NN$ with $D_\ell(\cX) \leq n+2$ and points $X^\cX_1, \ldots, X^\cX_{K(\cX)} \in \fN_{+,\ell}(\cX)$ so that:
\begin{enumerate}
\item[(a)] $\fN_{+,\ell}(\cX) \subset \cup_{i=1}^{K(\cX)} P(X^\cX_i,\gamma^{\ell+1})$ with $K(\cX) \leq C \gamma^{-D_\ell(\cX)}$; 
\end{enumerate}
and denoting $\iota_\ell(\cX) = (\bx,t)$ and $X^\cX_i=(\bx_i,t_i)$ we also have:
\begin{enumerate}
\item[(b)] if $\cM' \in \cup_{s \in S} \fF(M_s)$ and $\fN_+ \cap P_\ell(\cX)\cap P(X^\cX_i,\gamma^{\ell+1})$ intersects $\sing \cM'$, then
\begin{align*}
\cR_{2\gamma^{\ell}\rho_0,2 \gamma^{\ell}r_0}\cM'(\bx,t - 4\gamma^{2\ell}) & \neq 0,\\
\cR_{2\gamma^{\ell+1}\rho_0,2 \gamma^{\ell+1}r_0}\cM'(\bx_i, t_i - 4\gamma^{2(\ell+1)}) & \neq 0, \text{ and},
\end{align*}
\item[(c)] if $\cM', \cM'' \in \cup_{s \in S} \fF(M_s)$ and $\fN_+ \cap P_\ell(\cX)\cap P(X^\cX_i,\gamma^{\ell+1})$ intersects both $\sing \cM'$ and $\sing \cM''$, then 
\begin{multline*}
\gamma^{-(\kappa+D_\ell(X))} \cdot (2\gamma^\ell)^{-1} d(\cR_{2\gamma^{\ell}\rho_0,2 \gamma^{\ell}r_0}\cM'(\bx,t - 4\gamma^{2\ell}), \supp\cM''(t-4\gamma^{2\ell})) \\
\leq (2\gamma^{\ell+1})^{-1}d(\cR_{2\gamma^{\ell+1}\rho_0,2 \gamma^{\ell+1}r_0}\cM'(\bx_i, t_i - 4\gamma^{2(\ell+1)}), \supp\cM''(t_i - 4\gamma^{2(\ell+1)})).
\end{multline*}
\end{enumerate}
\end{corollary}

\begin{definition}
Let us now define the $\cT[\ell]$, $\ell \in \{ 0, 1, 2, \ldots \}$ and parent maps $\fp$:
\begin{itemize}
	\item For $\ell=0$, take $\cT[0] \subset \fN_+$ an arbitrary finite set so that $\fN_+ \subset \cup_{X \in \cT[0]} P(X,1)$.  Let $\iota_0$ denote the inclusion (identity) map.
	\item Then, inductively define
	\[ \cT[\ell+1] : = \bigcup_{\cX \in \cT[\ell]} \{ X^\cX_1, \ldots, X^\cX_{K(\cX)} \} \times \{\cX\} \]
	and
	\[ \iota_{\ell+1} : \cT[\ell+1] \to \fN_+, \]
	\[ \fp : \cT[\ell+1]\to\cT[\ell], \]
	to be projection onto the first and second factors, respectively.
\end{itemize}
\end{definition}

\begin{corollary}\label{coro:tree-covers}
For $\ell\in\{0,1,\dots\}$ it holds that 
\[
\fN_+ \subset \bigcup_{\cX \in \cT[\ell]} P_\ell(\cX). 
\]
\end{corollary}
\begin{proof}
 We induct on $\ell$. When $\ell=0$, since $P_0(X) = P(\iota_0(X),1)$, the assertion follows from the definition of $\cT[0]$. If the assertion holds for $\ell$, then
 \[
 \fN_+ = \fN_+ \cap \bigcup_{\cX \in \cT[\ell]}P_\ell(\cX) = \bigcup_{\cX \in \cT[\ell]} \fN_{+,\ell}(\cX).
 \]
By  (a) in Corollary \ref{coro:appl-covering-separation-depth-ell} and our definition of $\fp$ and $\iota_{\ell+1}$,
\[
\fN_{+,\ell}(\cX) \subset P_\ell(\cX) \cap \bigcup_{i=1}^{K(\cX)} P(X^\cX_i,\gamma^{\ell+1}) = \bigcup_{\cX' \in \fp^{-1}(\cX)} P_{\ell+1}(\cX').
\]
Therefore, since $\cup_{ \cX \in \cT[\ell]} \fp^{-1}(\cX) = \cT[\ell+1]$,
\[ 
\fN_+ = \bigcup_{\cX \in \cT[\ell]} \fN_{+,\ell}(\cX) \subset \bigcup_{\cX \in \cT[\ell]} \bigcup_{\cX' \in \fp^{-1}(\cX)} P_{\ell+1}(\cX') = \bigcup_{\cX' \in \cT[\ell+1]} P_{\ell+1}(\cX'). 
\]
This proves the assertion. 
\end{proof}

\begin{proposition}\label{prop:seperation-estimate-s-param}
Suppose that $\cX \in \cT[\ell]$ and $\iota_\ell(X) \in \sing \cM$ for $\cM\in \fF(M_s)$, $s\in S$, and suppose that $\cM'\in\fF(M_{s'})$, $s'\in S$, is such that $\fN_{+,\ell}(\cX)$ intersects $\sing \cM' \neq \emptyset$. Then,
\[
|s'-s| < 16 r_0 \gamma^{\ell\kappa}\gamma^{\sum_{j=0}^{\ell-1}D_j(\fp^{(\ell-j)}(\cX))}  
\]
\end{proposition}
\begin{proof}
Write $\iota_{j}(\fp^{(\ell-j)}(\cX)) = (\bx_{j},t_{j})$ for $j=0,\dots,\ell$. Apply Corollary \ref{coro:appl-covering-separation-depth-ell} (c) at level $j$ (in place of $\ell$ there) centered at $\fp^{(\ell-j)}(\cX)$ (in place of $\cX$ there) and $(\bx_{j+1},t_{j+1})$ (in place of $\cX_i$ there) and $\cM$ (in place of $\cM''$ there) to obtain
\begin{multline}\label{eq:distance-growth-power-decay-to-iterate}
\gamma^{-(\kappa+D_{j}(\fp^{(\ell-j)}(\cX)))} (2\gamma^j)^{-1} d(\cR_{2\gamma^{j}\rho_0,2 \gamma^{j}r_0}\cM'(\bx_j,t_j - 4\gamma^{2j}), \supp\cM(t_j-4\gamma^{2j})) \\
\leq (2\gamma^{j+1})^{-1}d(\cR_{2\gamma^{j+1}\rho_0,2 \gamma^{j+1}r_0}\cM'(\bx_{j+1},t_{j+1} - 4\gamma^{2(j+1)}), \supp\cM(t_{j+1} - 4\gamma^{2(j+1)})).\end{multline}
Meanwhile, we note that at the smallest scale we have
\begin{equation}\label{eq:distance-growth-power-decay-to-iterate-smallest}
d(\cR_{2\gamma^{\ell}\rho_0,2 \gamma^{\ell}r_0}\cM'(\bx_{\ell},t_{\ell} - 4\gamma^{2\ell}), \supp\cM(t_{\ell} - 4\gamma^{2\ell})) < 4 \gamma^\ell r_0.
\end{equation}
Indeed, Corollary \ref{coro:appl-covering-separation-depth-ell} (b) implies that
\[ \cR_{2\gamma^{\ell}\rho_0,2 \gamma^{\ell}r_0}\cM'(\bx_{\ell},t_{\ell} - 4\gamma^{2\ell}) \cap B_{2\gamma^\ell r_0}(\bx_\ell) \neq \emptyset, \]
\[ \supp\cM(t_{\ell} - 4\gamma^{2\ell}) \cap B_{2\gamma^\ell r_0}(\bx_\ell) \neq \emptyset, \]
which in turn yields \eqref{eq:distance-growth-power-decay-to-iterate-smallest} by the triangle inequality.

Now, \eqref{eq:distance-growth-power-decay-to-iterate-smallest} and an iteration of \eqref{eq:distance-growth-power-decay-to-iterate} from $j=0$ to $j=\ell-1$ yield
\[
d(\cR_{2\rho_0,2r_0}\cM'(\bx_0,t_0 - 4), \supp\cM(t_0-4))  <  8 r_0 \gamma^{\ell \kappa}\gamma^{\sum_{j=0}^{\ell-1}D_j(\fp^{(\ell-j)}(\cX))}.
\]
Using that the distance between the support of co-dimension one Brakke flows is non-decreasing \cite[10.6]{Ilmanen:elliptic}, we get
\[
d(M_{s'}, M_s) =  d(\supp \cM'(0), \supp \cM(0))   < 8 r_0 \gamma^{\ell \kappa}\gamma^{\sum_{j=0}^{\ell-1}D_j(\fp^{(\ell-j)}(\cX))}.
\]
The result follows from the fact that the foliation speed of $\{ M_s \}$ implies
\[ |s'-s| \leq 2 d(M_{s'}, M_s) < 16 r_0 \gamma^{\ell \kappa}\gamma^{\sum_{j=0}^{\ell-1}D_j(\fp^{(\ell-j)}(\cX))} \]
by \eqref{eq:S-foliation-speed-assumpt}. This completes the proof.
\end{proof}

\subsection{The proof of Proposition \ref{prop:dens-fS-drop}}

For $\ell \in\{0,1,\dots\}$ fixed, we form an open cover of $S$ as follows: for $\cX \in \cT[\ell]$ we fix $s\in S$ with $\iota_\ell(\cX) \in \sing \cM$ for some $\cM \in \fF(M_s)$. Then, let $U_\ell(\cX) \subset S$ denote the relatively open interval centered at $s$ with length 
\[
|U_\ell(\cX)| = 32 r_0 \gamma^{\ell\kappa}\gamma^{\sum_{j=0}^{\ell-1}D_j(\fp^{(\ell-j)}(\cX))},
\]
or if $s$ is close to $\partial S$, then one part of $U_\ell(X)$ will be cut off and this estimate will be replaced by $\leq$; this goes in the right direction below. 

We now note that
\[ S = \bigcup_{\cX \in \cT[\ell]} U_\ell(\cX). \]
Indeed, for $s' \in S$, \eqref{eq:dens-fS-not-much-change} implies that there is $\cM' \in \fF(M_{s'})$ and $X' \in \fN_+ \cap \sing \cM'$. By Corollary \ref{coro:tree-covers}, $X' \in P_\ell(\cX)$ for some $\cX \in \cT[\ell]$, so $s' \in U_\ell(\cX)$ by Proposition \ref{prop:seperation-estimate-s-param}. 

It follows that
\begin{align*}
|S| 
& \leq \sum_{\cX \in \cT[\ell]} |U_\ell(\cX)| \\
& \leq 32 r_0 2^\ell  \gamma^{\ell\kappa} \sum_{\cX \in \cT[\ell]}\gamma^{\sum_{j=0}^{\ell-1}D_j(\fp^{(\ell-j)}(\cX))}\\
& = 32 r_0 2^\ell  \gamma^{\ell\kappa} \sum_{\cX_0 \in \cT[0]} \sum_{\cX_1 \in \fp^{-1}(X_0)} \cdots \sum_{\cX_\ell \in\fp^{-1}(X_{\ell-1})} \gamma^{\sum_{j=0}^{\ell-1}D_j(\cX_j)}
\end{align*}
By Corollary \ref{coro:appl-covering-separation-depth-ell} and the definition of $\fp$ we have for $j \in \{ 0, 1, \ldots, \ell-1 \}$ that
\[ |\fp^{-1}(\cX_{j})| = K(\cX_j) \leq C \gamma^{-D_j(\cX_j)}. \]
Therefore,
\[ 
|S| \leq 32 r_0 |\cT[0]|  \gamma^{\ell\kappa} C^\ell  = 32 r_0 |\cT[0]| (C \gamma^{\kappa})^{\ell}.
\]
By \eqref{eq:choice-gamma}, it holds that $C \gamma^{\kappa}<1$. Thus, the right hand side is $o(1)$ as $\ell\to\infty$, so $|S| = 0$. This is a contradictionto the fact that $S$ was an interval with $|S| > 0$, completing the proof.

\section{Proof of Theorem \ref{thm:main}}\label{sec:proof-main}

We fix $n \in \{ 2, 3, \ldots \}$, $\Lambda \in (0, 2]$ so that \eqref{heart_cond} and \eqref{diamond_cond} hold, any closed embedded $M^n\subset \RR^{n+1}$ with $\lambda(M) \leq \Lambda$. We need some preparations for the proof.

\subsection{Low-entropy local foliation} \label{subsec:foliation}

Flowing $M$ by (smooth) mean curvature flow for a short time will produce $M'$ an arbitrarily small $C^\infty$ graph over $M$ with $\lambda(M') < \lambda(M) \leq \Lambda$ unless $M$ is a self-shrinker. If $M$ is a self-shrinker then either it is a round sphere (in which case Theorem \ref{thm:main} follows trivially) or, by \cite[Theorem 0.7]{ColdingMinicozzi:generic}, there is an arbitrarily small $C^\infty$-graph $M'$ with $\lambda(M') < \lambda(M) \leq \Lambda$. As such, up to replacing $M$ by $M'$ considered above, we can assume that
\[ \lambda(M) \leq \Lambda - 2\eps. \]
For $S$ a closed interval with $0\in S$ and $|S| > 0$, let $\{M_s\}_{s \in S}$ denote a unit-speed (local) foliation so that the $M_s$ are uniformly smooth and embedded with
\[ \lambda(M_s) \leq \Lambda -\eps \text{ for all } s \in S. \]

\subsection{The generic strata}

Recall the definitions of $\sing (S), \singgen (S), \singng (S) \subset \RR^{n+1} \times \RR$ from \eqref{eq:singS-ng-g} and define, for $k=0,\ldots,n-1$, the $k$-th generic stratum by
\[
\fG^k : = \{X \in \sing_\textrm{gen}(S) : \Theta(X) = \lambda(\SS^{n-k}\times \RR^k)\} \subset \RR^{n+1} \times \RR.
\]
Note that
\[ \singgen(S) = \cup_{k=0}^{n-1} \fG^k. \]
Fix $\eta_1=\eta_1(n)$ so that
if $X_\ell \in \singng(S)$ and $X_\ell \to X \in \fG^k$, $k = 0, 1, \ldots, n-1$, then 
		\[ \limsup_\ell \Theta(X_\ell) \leq \lambda(\SS^{n-k}) - \eta_1. \]
The existence of such $\eta_1$ follows from a straightforward adaptation of \cite[Theorem 0.2]{ColdingIlmanenMinicozzi} to the present setting.  

\subsection{Closed subsets of non-generic singular points}

For $\alpha>0$ define the following subsets of the set of non-generic singular points:
\[
\tilde \fN_\alpha : = \singng(S) \setminus U_\alpha(\singgen(S))
\]
and, for $k=0,\ldots,n-1$, 
\[
\fN_\alpha^k := \singng(S) \cap \overline{\fG^k} \setminus \cup_{j=k+1}^{n-1} U_\alpha(\fG^j).
\]
Here, $\bar{\cdot}$ denotes the closure and $U_\alpha(\cdot)$ the open $\alpha$-neighborhood in spacetime $\RR^{n+1}\times\RR$, with respect to the parabolic metric. Finally, define
\[ \fN_\alpha : = \tilde \fN_\alpha \cup \fN_\alpha^{n-1} \cup \cdots \cup \fN_\alpha^{0}. \]
\begin{lemma} \label{lemm:na-closed}
$\fN_\alpha \subset \RR^{n+1} \times \RR$ is closed.
\end{lemma}
\begin{proof}
Let's show that each set in the union is closed. The case of $\tilde \fN_\alpha$ is straightforward: if $X_\ell \in \tilde\fN_\alpha$ are such that $X_\ell \to X$, then clearly $X\not \in U_\alpha(\singgen(S))$, so $X \in \tilde\fN_\alpha$. 

Next, consider $\fN_\alpha^k$. Take $X_\ell \in \fN_\alpha^k$ with $X_\ell\to X$. Then, $X \in \overline{\fG^k} \setminus \cup_{j=k+1}^{n-1} U_\alpha(\fG^j)$ is immediate.  It remains to prove that $X \in \singng(S)$. 

If not, then $X \in \singgen(S)$, so $X \in \cup_{m=0}^k \fG^m$ since $X \not \in \cup_{j=k+1}^{n-1} U_\alpha(\fG^j)$. Note that
\begin{equation}\label{eq:fSalpha-closed-usc-density}
X_\ell, \; X \in \overline{\fG^k} \implies \Theta(X_\ell), \; \Theta(X) \geq \lambda(\SS^{n-k})
\end{equation}
by upper semicontinuity of density. Combining \eqref{eq:fSalpha-closed-usc-density} and $\lambda(\SS^1) > \lambda(\SS^2) > \dots > \lambda(\SS^n)$, we have that $X \in \fG^k$. Thus, by the choice of $\eta_1$, we have
\[
\limsup_\ell \Theta(X_\ell) \leq \lambda(\SS^{n-k}) - \eta_1.
\]
This contradicts the fact that $\Theta(X_\ell) \geq \lambda(\SS^{n-k})$ by \eqref{eq:fSalpha-closed-usc-density}.
\end{proof}

\subsection{The perturbation}

It follows from Lemma \ref{lemm:na-closed} that Corollary \ref{coro:dense-no-fS} applies with $\fN_\alpha$ in place of $\fN$ and yields a relatively open dense subset $S'_\alpha \subset S$ with $\sing(S'_\alpha) \cap \fN_\alpha = \emptyset$. 

Consider $S' = \cap_{\alpha \in \QQ\cap(0,1]} S'_\alpha$, which is dense by the Baire category theorem. Theorem \ref{thm:main} now follows from Lemma \ref{lemm:final-regularity-concl} below.

\begin{lemma}\label{lemm:final-regularity-concl}
$\singng(S')=\emptyset$.
\end{lemma}
\begin{proof}
Suppose, for contradiction, that there did exist some $X \in \singng(S')$, and take any sequence $\alpha_i \in \QQ \cap (0,1]$ with $\alpha_i\to 0$. Fix some $i$. Using that
\begin{equation} \label{eq:final-regularity-concl-subset}
	S' \subset S_{\alpha_i} \text{ and } \sing(S_{\alpha_i}) \cap \fN_{\alpha_i} = \emptyset,
\end{equation}
we see that $X \not \in \fN_{\alpha_i}$. But $\fN_{\alpha_i} \supset \tilde \fN_{\alpha_i}$, so $X \not \in \tilde\fN_{\alpha_i}$, so in particular $X \in U_{\alpha_i}(\singgen(S))$. Since $i$ was arbitrary,
\[
X \in \cap_i U_{\alpha_i} (\singgen(S)) = \overline{\singgen(S)} = \overline{\cup_{k=0}^{n-1} \fG^k}. 
\]
As a result, there exists some $k = 0, 1, \ldots, n-1$ so that
\begin{equation} \label{eq:final-regularity-concl-k}
	X \in \overline{\fG^k} \setminus \cup_{j=k+1}^{n-1} \overline{\fG^j}.
\end{equation}
Fix $i$ again. Using \eqref{eq:final-regularity-concl-subset} again, now together with $\fN_{\alpha_i} \supset \fN^k_{\alpha_i}$, it follows that $X \not \in \fN^k_{\alpha_i}$. Combined with $X \in \singng(S')$ we see that 
\[
X \notin \overline{ \fG^k} \setminus \cup_{j=k+1}^{n-1} U_{\alpha_i}(\fG^j).  
\]
Since $X \in \overline{\fG^k}$ by \eqref{eq:final-regularity-concl-k}, it must hold that
\[ X \in  \cup_{j=k+1}^{n-1} U_{\alpha_i}(\fG^j). \]
But since $i$ was arbitrary, this implies that
\[ X \in \cap_i  \cup_{j=k+1}^{n-1} U_{\alpha_i}(\fG^j) = \cup_{j=k+1}^{n-1} \overline{\fG^j}, \]
contradicting \eqref{eq:final-regularity-concl-k}. This completes the proof. 
\end{proof}

\appendix
\section{The generic strong multiplicity-one property}\label{app:strong-mult-one}
In this appendix we consider the following condition:

\begin{definition} \label{defi:strong.mult.one}
We say that a closed embedded hypersurface $M^n \subset \RR^{n+1}$ satisfies the \emph{strong multiplicity-one property} if the following holds. Consider:
\begin{itemize}
\item hypersurfaces $M_j \subset\RR^{n+1}$ converging smoothly to $M$,
\item flows $\cM_j \in \fF(M_j)$,
\item space-time points $X_j \to X \in \RR^{n+1}\times (0,\infty)$, and 
\item scales $\lambda_j \to \infty$,
\end{itemize}
so that $\operatorname{ParDil}_{\lambda_j}(\cM_j - X_j) \rightharpoonup \tilde\cM$. Then $\reg\tilde\cM$ has multiplicity one and the parabolic dimension of the singular set satisfies $\dim_H \sing\tilde\cM \leq n$. 
\end{definition}

In particular, polyhedral cones other than the flat multiplicity one hyperplane cannot arise as limit flows of perturbations of $M$.

\begin{remark}
Bamler has recently proven that a condition of similar to the strong multiplicity-one property does indeed hold in the Ricci flow setting \cite{Bamler:compactness,Bamler:structure} (in all dimensions). 
\end{remark}

\begin{remark} \label{rema:strong.mult.one.r3}
Bamler--Kleiner’s \cite{BamlerKleiner} recent resolution of Ilmanen’s multi\-plicity-one conjecture implies that the strong multiplicity-one property in Definition \ref{defi:strong.mult.one} holds for any embedded $M^2 \subset \RR^3$.
\end{remark}

It's important to observe that by dimension reduction \cite{White:stratification} (and the cyclic property of $M$ \cite{White:cyclic}), if $M$ fails to have the strong multiplicity-one property then we can adjust $X_j,\lambda_j$ so as to ensure that either:
\begin{enumerate}
\item $\tilde \cM$ is a static/quasi-static multiplicity $\geq 2$ hyperplane, or
\item a union of $\geq 4$ multiplicity-one half $n$-planes meeting along an $(n-1)$-plane. 
\end{enumerate}
(Thus, e.g., all $M$ with $\lambda(M) < 2$ automatically have the strong multiplicity-one property; cf. Lemma \ref{lemm:sing-F-stat}.)

\begin{lemma}[Openness of the strong multiplicity-one property]\label{lemm:open-strong-mult}
Suppose that $M_k$ are smooth closed embedded hypersurfaces in $\RR^{n+1}$ smoothly converging to $M$. If $M$ satisfies the strong multiplicty-one property then so does $M_k$ for $k$ sufficiently large. 
\end{lemma}
\begin{proof}
If not, there's $M_{k,j}$ smoothly converging to $M_k$, flows $\cM_{k,j}\in \fF(M_{k,j})$, space-time points $X_{k,j}\to X_k \in \RR^{n+1}\times (0,\infty)$, and scales $\lambda_{k,j}\to\infty$ so that 
\[
\operatorname{ParDil}_{\lambda_{k,j}}(\cM_{k,j} - X_{k,j}) \rightharpoonup \tilde\cM_k
\] 
as $j\to\infty$, where either $\tilde\cM_k$ satisfies (1) or (2) above. If (1) is satisfied for infinitely many $k$ we can pass to a subsequence so that $\tilde\cM_k \rightharpoonup \tilde\cM$ still with $\tilde\cM$ satisfying (1). If (2) is satisfied for infinitely many $k$, then either $\tilde\cM_k \rightharpoonup \tilde\cM$ still with $\tilde\cM$ satisfying (2) as is, or (2) where or two (or more) of the half-planes overlap.

In any case, we can (possibly after adjusting $X_{k,j},\lambda_{k,j}$ and passing to an un-labeled subsequence) find a sequence $j_0(k)\to \infty$ so that if $j(k)\geq j_0(k)$ then 
\[
\operatorname{ParDil}_{\lambda_{k,j(k)}}(\cM_{k,j(k)} - X_{k,j(k)}) \rightharpoonup \tilde\cM
\] 
where $\tilde\cM$ satisfies either (1) or (2). Indeed, the only case that's not obvious is the final possibility in which $\tilde\cM_k$ is a union of half-planes with two or more converging. In this case we just have to re-center the points and scales in the converging half-planes so as to arrange that we're in case (1). 

Taking $j(k) \geq j_0(k)$ sufficiently large, we can ensure that $M_{k,j(k)}$ converges smoothly to $M$.  In particular,  the flows $\cM_{k,j(k)}$ are smooth for a definite interval of time, and thus the points $X_{k,j(k)}$ are bounded away from time $0$. The data $M_{k,j(k)}$, $\cM_{k,j(k)}$, $X_{k,j(k)}$, $\lambda_{j(k)}$ contradict the assumed strong multiplicity-one property of $M$. This completes the proof. 
\end{proof}

We will work with initial data that can be approximated by hypersurfaces with the strong multiplicity-one property:

\begin{definition}
We say that a closed embedded hypersurface  $M^n\subset\RR^{n+1}$ satisfies the \emph{generic strong multiplicity one-property} if there are $M_k$ smoothly converging to $M$ with $M_k$ having the strongly multiplicity-one property. 
\end{definition}

The techniques in this paper used to prove Corollaries \ref{coro:R4} and \ref{coro:R5} can be used to prove the following result, which drops all entropy bounds and replaces them with the generic strong multiplicity-one assumption. (A stronger partial result holds for $n=2$ by \cite{CCMS:generic1,CCS:generic2}; see \cite[Theorem 9.2]{CCS:generic2} for a precise statement.) 

\begin{theorem}\label{theo:generic-mult-one-assump}
If $M^n \subset \RR^{n+1}$, $n\in \{2,3,4\}$ has the generic strong multiplicity one property then there exist arbitrarily small $C^\infty$ graphs $M'$ over $M$ so that $\singng\cM'=\emptyset$ for all $\cM'\in\fF(M')$. 
\end{theorem}

Note that a natural generalization of Ilmanen's multiplicity-one conjecture would be for all \emph{all} initial data themselves to have the strong multiplicity-one property directly; a weaker one is for all initial data to satisfy the \emph{generic} strong multiplicity-one property.

Let us explain the necessary modifications to prove Theorem \ref{theo:generic-mult-one-assump}. First, by assumption we can perturb $M$ (not relabeled) so that $M$ satisfies the strong multiplicity-one property. Now, following the setup in Section \ref{sec:proof-main}, we embed $M$ in a local unit speed foliation $\{M_s\}_{s\in S}$. By Lemma \ref{lemm:open-strong-mult} we can shrink $S$ (still with $0\in S$) so that all $M_s$, $s \in S$, have the strong multiplicity-one property. 

We make the following modifications to the argument used in the rest of the paper:
\begin{enumerate}
\item Any $\cM \in \fF(M_s)$, $s\in S$ has $\reg\cM$ connected and multiplicity one by \cite[Corollary F.4]{CCMS:generic1} and the strong-multiplicity one property (here we just have to consider tangent flows to $\cM$ without varying the parameter $s$). In particular, the regularity scale continues to be continuous (cf. Remark \ref{rema:reg-scale-continuity}). 
\item Instead of considering arbitrary $F$-stationary varifolds we always consider $F$-stationary varifolds that correspond to self-similar flows arising as limits of $\cM_{s_j}\in\fF(M_{s_j})$, $s_j \in S$. The conclusion of Lemma \ref{lemm:sing-F-stat} continues to hold in this setting. The $\lambda(V) < 2$ assumption in subsequent results (e.g. Lemma \ref{lemm:smooth-cross-stat-var}, Corollary \ref{coro:dont-cross-dens-well-defined}, Proposition \ref{prop:finite-mu-smooth}, etc.) is simply used to refer to a regularity result along the lines of Lemma \ref{lemm:sing-F-stat}, and thus these results continue to hold in this more general setting. 
\item The class $\sP(\eta)$ as in Definition \ref{defi:Peta0} should be replaced with the smaller set of such pairs that arise as rescalings of flows $\cM \in \fF(M_s)$, $s\in S$ (and condition (3) should be removed). A similar modification (only considering flows that arise as rescalings of flows starting at the $M_s$, $s\in S$) should be made in subsequent arguments (e.g.\ Lemmas  \ref{lemm:rescaled-flows-converge-graphical} and \ref{lemm:unif-Harnack}, as well as Proposition \ref{prop:main-cover-sep}). 
\end{enumerate}
The remainder of the argument is unchanged. 
\begin{remark}
It would suffice to assume that $M$ satisfies a different  form of the generic strong multiplicity-one property in which we only considers initial data that are leaves of the local foliation $\{M_s\}$.
\end{remark}

\begin{remark}
Note that we do not know if the second main hypothesis \eqref{heart_cond} holds when $n=5$ and $\Lambda$ is large (in other words, we do not know if there's a non-generic shrinker $\Sigma^2\subset \RR^3$ with $\mu(\Sigma) \geq -\frac32$). Because of this, it's unclear if Theorem \ref{theo:generic-mult-one-assump} holds in $\RR^6$ (even assuming the generic strong multiplicity one property for $M$).  
\end{remark}

\section{Stable minimal cones as shrinkers}\label{app:stable-mu}

In this appendix, we prove an estimate for the first eigenvalue of the $L$-operator on a stable minimal cone. This is not used elsewhere in the paper. This should be related to the question of whether or not the Simons cone can arise generically as a (static) tangent flow (cf.\ \cite{Velazquez,GuoSesum,Stolarski}).

Lemma \ref{prop:finite-mu-smooth} has the following consequence.

\begin{corollary}\label{coro:unstable-cone}
Let $V$ be a conical stationary cyclic integral $n$-varifold in $\RR^{n+1}$ with $\lambda(V) < 2$. If $\reg V$ is unstable with respect to the Euclidean area functional, then $\mu(V) = -\infty$.
\end{corollary}

If $\reg V$ is stable, the conclusion of Corollary \ref{coro:unstable-cone} need not hold. For example it is easy to check that $\mu(\RR^n) = -\tfrac 12$. We now show that non-flat stable cones satisfy an improved inequality. 
\begin{proposition}
Take $n\geq 2$ and $V$ any a conical stationary cyclic integral $n$-varifold in $\RR^{n+1}$ with\footnote{Actually, the result here would hold for a stationary minimal cone satisfying Wickramasekera's $\alpha$-structural hypothesis \cite{Wickramasekera:SS} (it is easy to check that a stationary integral cyclic varifold $V$ with $\lambda(V) < 2$ satisfies the $\alpha$-structural hypothesis).} $\lambda(V) < 2$. If $V$ is not a flat hyperplane, then
\[ \mu(V) \leq -1-\delta_n, \]
for some $\delta_n>0$ depending only on $n$. 
\end{proposition}
\begin{proof}
By \cite[Lemma 6.1.1]{Simons:minvar} and dimension reduction we can assume that $n\geq 7$ (otherwise Corollary \ref{coro:unstable-cone} gives $\mu(V) = -\infty$). 

Let $\Sigma \subset \SS^n$ denote the regular part of the link of $V$. By \cite{Zhu:sphere-stability} (cf.\ \cite{Simons:minvar}) we can find $\Sigma' \Subset \Sigma$ so that $\Sigma'$ has smooth boundary and so that there is $u \in C^\infty(\Sigma')$ with $u=0$ on $\partial\Sigma'$ (if nonempty) and
\begin{equation}\label{eq:first-dir-eig-link-cone}
\Delta_\Sigma u + ( |A_\Sigma|^2 - (n-1)) u \geq 0. 
\end{equation}
Indeed, by \cite[Theorem 0.1]{Zhu:sphere-stability} (since $\Sigma$ is not a totally geodesic $\SS^n$ by assumption) it holds that
\[
\inf\left\{ \int_\Sigma |\nabla u|^2 + (|A_\Sigma|^2 - (n-1)) u^2 : u \in C^\infty_c(\Sigma), \int_\Sigma u^2 = 1\right\} \leq 0
\]
with equality if and only if $V$ is the cone over a Clifford hypersurface (and in particular $\sing V =\{\bOh\}$). As such, if $\sing V \neq \{\bOh\}$ this equality is strict, so we can exhaust $\Sigma$ by smooth regions; the first Dirichlet eigenfunction on these regions will eventually satisfy \eqref{eq:first-dir-eig-link-cone}. Conversely, if $\sing V = \{\bOh\}$ we can find $u$ satisfying \eqref{eq:first-dir-eig-link-cone} with $\Sigma' =\Sigma$ (in this case, this is a consequence of \cite[Lemma 6.1.7]{Simons:minvar}). 

We now consider $f(r,\omega) = w(r)u(\omega)$ for $u$ as in \eqref{eq:first-dir-eig-link-cone} and $w \in C^\infty_c((0,\infty))$ in \eqref{eq:first-eig-L}. We find
\begin{align*}
\mu(V) \int_0^\infty w(r)^2 e^{-\tfrac 14 r^2} r^{n-1} dr 
& \leq \int_0^\infty \left( w'(r)^2 - ((n-1)r^{-2} + \tfrac 12) w(r)^2 \right) e^{-\tfrac 14 r^2} r^{n-1} dr.
\end{align*}
A standard argument shows that for any $\alpha > \frac{2-n}{2}$ we can take
\[
w(r) = \begin{cases}
0 & r \in (0,\eps^2]\\
\eps^\alpha(2-\tfrac{\log r}{\log \eps}) & r \in (\eps^2,\eps]\\
r^\alpha & r \in (\eps,\eps^{-1}]\\
\eps^{-\alpha}(2-\eps r) & r \in (\eps^{-1},2\eps^{-1}]\\
0 & r \in (\eps^{-2},\infty). 
\end{cases}
\]
This yields
\begin{align*}
& \mu(V) \int_0^{\infty} e^{-\tfrac 14 r^2} r^{n+2\alpha-1} dr \\
& \leq \int_0^\infty \left( (\alpha^2-(n-1))r^{-2} -   \tfrac 12  \right) e^{-\tfrac 14 r^2} r^{n+2\alpha-1} dr + o(\eps^{n+2\alpha-2}). 
\end{align*}
We perform the change of variables $r=2\sqrt{s}$ in both integrals:
\begin{align*}
& \mu(V) \int_0^{\infty} e^{-s} s^{\frac{n+2\alpha}{2}-1} ds \\
& \leq \frac 12  \int_0^\infty \left( \tfrac{\alpha^2-(n-1)}{2} s^{-1} - 1)  \right) e^{-s} s^{\frac{n+2\alpha}{2}-1} ds + o(\eps^{n+2\alpha-2}). 
\end{align*}
Recalling that for $\textrm{Re}(z) > 0$ the Gamma function $\Gamma(z) : = \int_0^\infty s^{z-1} e^{-s} ds$ satisfies the recurrence relation $\Gamma(z+1) = z\Gamma(z)$, we find 
\begin{align*}
\mu(V) & \leq \frac 12 \left(\frac{\alpha^2-(n-1)}{2}  \frac{\Gamma(\frac{n+2\alpha}{2} - 1)}{\Gamma(\frac{n+2\alpha}{2})} - 1 \right) + o(\eps^{n+2\alpha-2})\\
& = \frac 12 \left(\frac{\alpha^2-(n-1)}{n+2\alpha-2}  - 1 \right) + o(\eps^{n+2\alpha-2})\\
& = \frac 12 \frac{\alpha^2 - 2\alpha -2n + 3}{n+2\alpha-2}  + o(\eps^{n+2\alpha-2})
\end{align*}
Thus, we find that
\[
\mu(V) \leq \inf_{\alpha > \frac{2-n}{2}}\frac 12 \frac{\alpha^2 - 2\alpha -2n + 3}{n+2\alpha-2}. 
\]
(Note that for $n \leq 6$ this infimum is $-\infty$; compare this with Corollary \ref{coro:unstable-cone}.) Furthermore,
\begin{align*}
\inf_{\alpha > \frac{2-n}{2}}\frac 12 \frac{\alpha^2 - 2\alpha -2n + 3}{n+2\alpha-2} & = \frac 1 4 \left( - n + \sqrt{8-8n+n^2} \right) \\
& = -1 - \underbrace{\frac 1 4 \left(   \sqrt{16-8n+n^2}  -  \sqrt{8-8n+n^2} \right)}_{:=\delta_n}
\end{align*}
(since $n\geq7$). This completes the proof. \end{proof}
Note that $\delta_7 =\tfrac 12 $, $\delta_8 \approx 0.29$, and $\delta_n \to 0$ as $n\to \infty$. 

\bibliographystyle{alpha}
\bibliography{bib}

\end{document}